\documentclass{article}

\usepackage{amsthm, amssymb, amsmath}
\usepackage{verbatim}
\usepackage{graphics}
\usepackage{graphicx}
\usepackage{setspace}

\newtheorem{thm}{Theorem}
\newtheorem{prop}[thm]{Proposition}
\newtheorem{cor}[thm]{Corollary}
\newtheorem{lem}[thm]{Lemma}

\theoremstyle{plain}
\newtheorem*{defn}{Definition}

\newcommand{\R}{\mathbb{R}}
\newcommand{\floor}[1]{\left\lfloor #1 \right\rfloor}
\newcommand{\fpart}[1]{\left\{ #1 \right\}}
\newcommand{\supp}{\textrm{supp}}
\newcommand{\jac}[1]{P_{#1}^{(\alpha,\beta)}}
\newcommand{\njac}[1]{R_{#1}^{(\alpha,\beta)}}
\newcommand{\var}{\textrm{Var}}
\newcommand{\sd}[2]{\mathcal{D}^{#1}_{#2}}
\newcommand{\vol}{\textrm{Vol}}

\title{Small Designs for Path Connected Spaces and Path Connected Homogeneous Spaces}
\author{Daniel M. Kane\footnotemark[1]
}

\begin{document}
\maketitle

\begin{abstract}
We prove the existence of designs of small size in a number of contexts.  In particular our techniques can be applied to prove the existence of $n$-designs
on $S^{d}$ of size $O_d(n^{d}\log(n)^{d-1})$.
\end{abstract}

\footnotetext[1]{Stanford University Department of Mathematics, \texttt{dankane@math.stanford.edu}}

\section{Introduction}

Given a measure space $(X,\mu)$ and a set $f_1,\ldots,f_m:X\rightarrow \R$, \cite{designExistence} defines an \emph{averaging set} to be a finite set of points, $p_1,\ldots,p_N\in X$ so that
\begin{equation}\label{averaginSetEquation}
\frac{1}{N}\sum_{i=1}^N f_j(p_i) = \frac{1}{\mu(X)} \int_X f_j d\mu
\end{equation}
for all $1\leq j\leq m.$
The authors of \cite{designExistence} show that if $X$ is a path-connected topological space, $\mu$ has full support, and the $f_i$ are continuous that such sets necessarily exist.  In this paper, we study the problem of how small such averaging sets can be.  In particular, we define a \emph{design problem} to be the data of $X$, $\mu$ and the vector space of functions on $X$ spanned by the $f_j$. For a design problem, $D$, we show that there exist averaging sets (we call them designs) for $D$ with $N$ relatively small.

Perhaps the best studied case of the above is that of spherical designs, introduced in \cite{designDef}.  A spherical design on $S^d$ of strength $n$ is defined to be an averaging set for $X=S^d$ (with the standard measure) where the set of $f_j$ is a basis for the polynomials of degree at most $n$ on the sphere.  It is not hard to show that such a design must have size at least $\Omega_d(n^d)$ (proved for example in \cite{designDef}).  It was conjectured by Korevaar and Meyers that designs of size $O_d(n^d)$ existed.  There has been much work towards this Conjecture.  Wagner proved in \cite{bound1} that there were designs of size $O_d(n^{12d^4})$.  This was improved by Korevaar and Meyers in \cite{bound2} to $O_d(n^{(d^2+d)/2})$, by Bondarenko, and Viazovska in \cite{bound3} to $O_d(n^{2d(d+1)/(d+2)})$.  In \cite{SmSph}, Bondarenko, Radchenko, and Viazovska recently announced a proof of the full conjecture.

In this paper, we develop techniques to prove the existence of small designs in a number of contexts.  In greatest generality, we prove that on a path-connected topological space there exist designs to fool any set of continuous functions on $X$ of size roughly $MK$, where $M$ is the number of linearly independent functions, and $K$ is a measure of how badly behaved these functions are.  We also show that if in addition $X$ is a homogeneous space and the linear span of functions we wish to fool is preserved under the symmetry group of $X$ that $K\leq M$.  For example, this immediately implies strength-$n$ designs of size $O(n^{2d}/(d!)^2)$ on $S^d$.  It also implies the existence of small Grassmannian designs (see \cite{GrassmanDesigns} for the definition).  Generally, this result proves the existence of designs whose size is roughly the square of what we expect the optimal size should be.

With a slight modification of our technique, we can also achieve better bounds in some more specialized contexts.  In particular, in Section \ref{IntervalSec} we produce designs of nearly optimal size for beta distributions on the interval $[-1,1]$, and in Section \ref{SphereSec}, we prove the existence of strength-$n$ designs on $S^d$ of size $O_d(n^d \log(n)^{d-1})$, which is optimal up to a polylog factor.

In Section \ref{BasicConceptsSec}, we describe the most general setting of our work and some of the fundamental ideas behind our technique.  In Section \ref{PathConnSec}, we handle our most general case of path-connected spaces.  In Section \ref{TightnessSec}, we produce an example in which the upper bound for sizes of designs in the previous section is essentially tight.  In Section \ref{HomogeneousSec}, we study the special case of homogeneous spaces.  In Section \ref{IntervalSec}, we provide nearly optimal bounds for the size of designs for beta distributions on the interval.  In Section \ref{SphereSec}, we prove our bounds on the size of spherical designs.

\section{Basic Concepts}\label{BasicConceptsSec}

We begin by defining the most general notion of a design that we deal with in this paper.

\begin{defn}
A \emph{design-problem} is a triple $(X,\mu,W)$ where $X$ is a measure space with a positive measure $\mu$, normalized so that $\mu(X)=1$, and $W$ is a vector space of $L^1$ functions on $X$.

Given a design-problem $(X,\mu,W)$, a \emph{design} of size $N$ is a list of $N$ points (not necessarily distinct) $p_1,p_2,\ldots,p_N \in X$ so that for every $f\in W$,
\begin{equation}\label{designDefEqu}
\int_X f(x)d\mu(x) = \frac{1}{N}\sum_{i=1}^N f(p_i).
\end{equation}

A \emph{weighted design} of size $N$ is a set of points $p_1,p_2,\ldots,p_N\in X$ and a list of weights $w_1,w_2,\ldots,w_N \in [0,1]$ so that $\sum_{i=1}^N w_i = 1$ and so that for each $f\in W$,
\begin{equation}\label{weightedDesignDefEqu}
\int_X f(x)d\mu(x) = \sum_{i=1}^N w_i f(p_i)
\end{equation}
\end{defn}

For example, if $(X,\mu)$ is the $d$-sphere with its standard (normalized) measure, and $W$ is the space of polynomials of total degree at most $n$ restricted to $X$, then our notion of a design (resp. weighted design) corresponds exactly to the standard notion of a design (resp. weighted design) of strength $n$ on the $d$-sphere.

Note that a design is the same thing as a weighted design in which all the weights are $\frac{1}{N}.$

Notice that if we set $f(x)$ to be any constant function that the formulas in Equations \ref{designDefEqu} and \ref{weightedDesignDefEqu} will hold automatically.  Hence for a design problem it is natural to define the vector space $V$ of functions on $X$ to be the space of functions, $f$, in $W+\langle 1 \rangle$ so that $\int_X f(x)d\mu(x) = 0$.

\begin{lem}\label{VCriterionLem}
For a design-problem $(X,\mu,W)$ with $V$ as defined above, $p_1,p_2,\ldots,p_N$ is a design (resp. $p_1,p_2,\ldots,p_N,w_1,w_2,\ldots,w_N$ is a weighted design) if and only if for all $f\in V$, $\sum_{i=1}^N f(p_i) = 0$, (resp. $\sum_{i=1}^N w_i f(p_i) = 0$).
\end{lem}
\begin{proof}
Since any design can be thought of as a weighted design, it suffices to prove the version of this Lemma for weighted designs.  First assume that $\sum_{i=1}^N w_i f(p_i) = 0$ for each $f\in V$.  For every $g\in W$, letting $f(x) = g(x)-\int_X g(y)d\mu(y)$, $f\in V$.  Hence
\begin{align*}
0 & = \sum_{i=1}^N w_i\left(g(p_i)-\int_X g(y)d\mu(y)\right)\\
& = \sum_{i=1}^N w_ig(p_i) - \left(\sum_{i=1}^N w_i \right)\left(\int_X g(y)d\mu(y)\right) \\
& = \sum_{i=1}^N w_ig(p_i) - \int_X g(x)d\mu(x).
\end{align*}
Hence $p_i,w_i$ is a weighted design.

If on the other hand, $p_i,w_i$ is a weighted design and $f\in V$, then $f(x) = g(x) + c$ for some $g\in W$ and constant $c$.  Furthermore $0 = \int_X g(x)+c d\mu(x) = \int_X g(x)d\mu(x) + c$ so $c=-\int_X g(x)d\mu(x).$  Hence
\begin{align*}
\sum_{i=1}^N w_if(p_i) & = \sum_{i=1}^N w_i(g(p_i) + c)\\
& = \sum_{i=1}^N w_i g(p_i) + \left(\sum_{i=1}^N w_i \right) c \\
& = \int_X g(x)d\mu(x) + c\\
&= 0.
\end{align*}
\end{proof}

It will also be convenient to associate with the design problem $(X,\mu,W)$ the number $M=\dim(V)$.  We note that there is a natural map $E:X\rightarrow V^*$, where $V^*$ is the dual space of $V$.  This is defined by $(E(p))(f) = f(p)$.  This function allows us to rephrase the idea of a design in the following useful way:

\begin{lem}\label{ECriterionLem}
Given a design problem $(X,\mu,W)$ along with $V$ and $E$ as described above, $p_i$ is a design (resp. $p_i,w_i$ is a weighted design) if and only if $\sum_{i=1}^N E(p_i) = 0$ (resp. $\sum_{i=1}^N w_i E(p_i)=0$).
\end{lem}
\begin{proof}
Again it suffices to prove only the version of this Lemma for weighted designs.  Note that for $f\in V$, that
$$
\sum_{i=1}^N w_i f(p_i) = \sum_{i=1}^N w_i (E(p_i))(f) = \left( \sum_{i=1}^N w_i E(p_i)\right)(f).
$$
This is 0 for all $f\in V$, if and only if $\sum_{i=1}^N w_i E(p_i)=0$.  This, along with Lemma \ref{VCriterionLem}, completes the proof.
\end{proof}

To demonstrate the utility of this geometric formulation, we present the following Lemma:

\begin{lem}\label{WeightedDesignLem}
Given a design problem $(X,\mu,W)$ with $V,M,E$ as above, if $M<\infty$, there exists a weighted design for this problem of size at most $M+1$.
\end{lem}
\begin{proof}
Note that for $f\in V$ that
$$\left(\int_X E(x)d\mu(x)\right)(f) = \int_X f(x)d\mu(x) = 0.$$
Therefore $\int_X E(x) d\mu(x) = 0.$  Therefore 0 is in the convex hull of $E(X)$.  Therefore 0 can be written as a positive affine linear combination of at most $M+1$ points in $E(X)$.  By Lemma \ref{ECriterionLem}, this gives us a weighted design of size at most $M+1$.
\end{proof}

Unfortunately, our notion of a design problem is too general to prove many useful results about.  We will therefore work instead with the following more restricted notion:

\begin{defn}
A \emph{topological design problem} is a design problem, $(X,\mu,W)$ in which $X$ is a topological space, the $\sigma$-algebra associated to $\mu$ is Borel, the functions in $W$ are bounded and continuous, and $W$ is finite dimensional.

We call a topological design problem \emph{path-connected} if the topology on $X$ makes it a path-connected topological space.

We call a topological design problem \emph{homogeneous} if for every $x,y\in X$ there is a measure-preserving homeomorphism $f:X\rightarrow X$ so that $f^*(W)=W$ and $f(x)=y$.
\end{defn}

We will also want a measure on the complexity of the functions in $W$ for such a design problem.

\begin{defn}
Let $(X,\mu,W)$ be a topological design problem.  Associate to it the number
$$
K = \sup_{f\in V\backslash \{0 \}}\frac{\sup(f)}{|\inf(f)|} =\sup_{f\in V\backslash \{0 \}}\frac{\sup(-f)}{|\inf(-f)|} =\sup_{f\in V\backslash \{0 \}}\frac{-\inf(f)}{\sup(f)} = \sup_{f\in V\backslash \{0 \}}\frac{\sup(|f|)}{\sup(f)}.
$$
\end{defn}

Notice that since $\frac{\sup(f)}{|\inf(f)|}$ is invariant under scaling of $f$ by positive numbers, and since $V\backslash \{0 \}$ modulo such scalings is compact, that $K$ will be finite unless there is some $f\in V\backslash \{0\}$ so that $f(x)\geq 0$ for all $x$.  Since $\int_X f(x)d\mu(x) = 0$ this can only be the case if $f$ is 0 on the support of $\mu$.

Throughout the rest of the paper, to each topological design problem, $(X,\mu,W)$ we will associate $V,E,M,K$ as described above.

\section{The Bound for Path Connected Spaces}\label{PathConnSec}

In this Section, we prove the following Theorem, which will also be the basis for some of our later results.

\begin{thm}\label{PathConnThm}
Let $(X,\mu,W)$ be a path-connected topological design problem.  If $M>0$, then for every integer $N > (M-1)(K+1)$ there exists a design of size $N$ for this design problem.
\end{thm}

Throughout the rest of this Section, we use $X,\mu,W,V,E,M,K,N$ to refer to the corresponding objects in the statement of Theorem \ref{PathConnThm}.  Our proof technique will be as follows.  First, we construct a convex polytope $P$ given by the convex hull of points of $E(X)$, that also contains the origin.  Next, we construct a continuous function $F:P\rightarrow V^*$ so that every point in the image of $F$ is a sum of $N$ points in $E(X)$, and so that for each facet, $T$, of $P$, $F(T)$ lies on the same side of the hyperplane through the origin parallel the one defining $T$ as $T$ does.  Lastly, we show, using topological considerations, that $0$ must be in the image of $F$.  We begin with the construction of $P$.

\begin{prop}\label{PConstructionProp}
For every $\epsilon>0$, there exists a polytope $P\subset V^*$ spanned by points in $E(X)$ such that for every linear inequality satisfied by the points of $P$ of the form
$$
\langle x, f\rangle \leq c
$$
for some $f\in V\backslash \{ 0 \}$, we have
$$
\sup_{p\in X} |f(p)| \leq c(K+\epsilon).
$$
\end{prop}
\begin{proof}
Suppose that $P$ is the the convex hull of some set of points $E(p_i)$ for some points $p_i\in X$.  Then it is the case that $\langle x, f\rangle \leq c$ for all $x\in P$ if and only if this holds for all $x=E(p_i)$, or if $f(p_i) \leq c$ for all $i$.  Hence it suffices to find some finite set of $p_i\in X$ so that for each $f\in V\backslash \{0\}$, $\sup(|f|) \leq \sup_i f(p_i)(K+\epsilon).$  Notice that this condition is invariant under scaling $f$ by a positive constant, so it suffices to check for $f$ on the unit sphere of $V$.

Notice that by the definition of $K$, that for each such $f$, there is a $p\in X$ so that $\sup(|f|)\leq f(p)K$.  Notice that for such a $p$, $\sup(|g|)\leq g(p)(K+\epsilon)$ for all $g$ in some open neighborhood of $f$.  Hence these $p$ define an open cover of the unit ball of $V$, and by compactness there must exist a finite set of $p_i$ so that for each such $f$, $\sup(|f|) \leq f(p_i)(K+\epsilon)$ for some $i$.  This completes our proof.
\end{proof}

Throughout the rest of this section we will use $\epsilon$ and $P$ to refer to a positive real number and a polytope in $V^*$ satisfying the conditions from Proposition \ref{PConstructionProp}.  We now construct our function $F$.

\begin{prop}\label{FConstructionProp}
If $\epsilon < \frac{N}{M-1} - K - 1$, there exists a continuous function $F:P\rightarrow V^*$ so that
\begin{itemize}
\item For each $x\in P$ there are points $q_1,\ldots,q_N\in X$ so that $F(x) = \sum_{i=1}^N E(q_i)$
\item For each facet, $T$, defined by the equation $L(x)=c$ for some linear function $L$ on $V^*$ and some $c\in\mathbb{R^+}$, $L(F(T))\subset \R^+$
\end{itemize}
\end{prop}
\begin{proof}
For a real number $x$, let $\floor{x}$ denote the greatest integer less than or equal to $x$ and let $\fpart{x}=x-\floor{x}$ denote the fractional part of $x$.

Let $p_i$ be points in $X$ so that $P_i=E(p_i)$ are the vertices of $P$.  Let $p_0$ be some particular point in $X$.  Since $X$ is path-connected, we can produce continuous paths $\gamma_i:[0,1]\rightarrow X$ so that $\gamma_i(0)=p_0$ and $\gamma_i(1)=p_i$.  For $r\in [0,1]$ a real number, we use $[rP_i]$ to denote $E(\gamma_i(r))$.  We let $[0]:=[0P_i]=E(p_0)$.  We also note that $[P_i]:=[1P_i]=P_i$ and that $[rP_i]$ is continuous in $r$.

Next pick a triangulation of $P$.  Our basic idea will be as follows: for any $Q\in P$, if $Q$ is in the simplex in our triangulation defined by $P_{n_0},P_{n_1},\ldots,P_{n_d}$ for some $n_i$ and $d\leq M$ we can write $Q$ uniquely as $\sum_{i=0}^d x_i [P_{n_i}]$ for $x_i\in [0,1]$ with $\sum_i x_i = 1$ (here we think of the sum as being a sum of points in $V^*$).  The idea is that $F(Q)$ should be approximately $NQ=\sum_{i=0}^d Nx_i [P_{n_i}]$.  If the $Nx_i$ are all integers, this is just a sum of $N$ points.  Otherwise, we need to smooth things out some, and define $F$ as follows.

Let $S$ be the set of $i\in\{0,\ldots,d\}$ so that $\fpart{Nx_i}\geq 1-1/(3M)$.  Define
\begin{align*}
F(x) := \sum_{i=0}^d \left(\floor{Nx_i}\right)[P_{n_i}] & + \sum_{i\in S} [(1-3M(1-\fpart{Nx_i}))\cdot P_{n_i}]\\
& + \left(N-\sum_{i=0}^d \floor{Nx_i} - |S|\right)[0].
\end{align*}

We have several things to check.  First, we need to check that $F$ is well defined.  Next, we need to check that $F$ is continuous.  Finally, we need to check that $F$ has the desired properties.

We must first show that $F$ is well defined.  We have defined it on each simplex of our triangulation, but we must show that these definitions agree on the intersection of two simplices.  It will be enough to check that if $Q$ is in the simplex defined by $P_{n_0},\ldots,P_{n_d}$ and the simplex defined by $P_{n_0},\ldots,P_{n_d},P_{n_{d+1}}$, that our two definitions of $F(Q)$ agree (because then all definitions of $F(Q)$ agree with the definition coming from the minimal simplex containing $Q$).  In this case, if we write $Q=\sum_{i=0}^d x_i P_{n_i} = \sum_{i=0}^{d+1}y_i P_{n_i}$, then it must be the case that $x_i=y_i$ for $i\leq d$ and $y_{d+1}=0$.  It is easy to check that our two definitions of $F$ on this intersection agree on $Q$.

To prove continuity, we need to deal with several things.  Firstly, since $F$ can be defined independently on each simplex in our decomposition of $P$ in such a way that the definitions agree on the boundaries, we only need to check that $F$ is continuous on any given simplex.  In this case, we may write $F(Q)=F(x_0,\ldots,x_d)$.  We also note that we can write $F(Q) = N[0] + \sum_{i=0}^d F_i(Nx_i)$ where $F_i(y)$ is
$$
\begin{cases} (\floor{y})\cdot( [P_{n_i}] - [0] ) \ \ & \textrm{if} \ \fpart{y}<1-1/(3M) \\
(\floor{y})\cdot( [P_{n_i}] - [0] ) + [(1-3M(1-\fpart{y}))\cdot P_{n_i}]  - [0] \ & \textrm{else} \end{cases}.
$$
We now have the check continuity of $F_i$.  Note that $F_i$ is clearly continuous except where $y$ is either an integer or an integer minus $1/(3M)$.  For integer $n$, as $y$ approaches $n$ from below, $F_i(y) = (n-1)([P_{n_i}]-[0])+[(1-3M(n-y))\cdot P_{n_i}]-[0] \rightarrow n([P_{n_i}]-[0]) = F_i(n)$.  Also as $y$ approaches $n-1/(3M)$ from below, $F_i(y) = (n-1)([P_{n_i}]-[0])=F_i(n-1/(3M))$.  Hence $F$ is continuous.

Next we need to check that for any $Q$ that $F(Q)$ is a sum of $N$ elements of $E(X)$.  From the definition it is clear that $F(Q)$ is sum of elements of $E(X)$ with integer coefficients that add up to $N$.  Hence, we just need to check that all of these coefficients are positive.  This is obvious for all of the coefficients except for $N-|S|-\sum_{i=0}^d \floor{Nx_i}$.  Hence, we need to show that
$N\geq |S|+\sum_{i=0}^d \floor{Nx_i}$.  Since $\sum_{i=0}^d x_i = 1$ by assumption,
\begin{align*}
N & = \sum_{i=0}^d Nx_i\\
& = \sum_{i=0}^d \floor{Nx_i}+\fpart{Nx_i}\\
& \geq \sum_{i=0}^d \floor{Nx_i} + \sum_{i\in S} \fpart{Nx_i} \\
& \geq \sum_{i=0}^d \floor{Nx_i} + |S|(1-1/(3M))\\
& = |S| + \sum_{i=0}^d \floor{Nx_i} - |S|/(3M).
\end{align*}
Since $N$ and $|S| + \sum_{i=0}^d \floor{Nx_i}$ are both integers and $|S|/(3M) \leq (M+1)/(3M) < 1$, this implies that $N\geq |S| + \sum_{i=0}^d \floor{Nx_i}$.

Finally, suppose that $T$ is some facet of $P$ defined by $L(x)=c>0$ and that $Q$ lies on $T$.  Since $(V^*)^* = V$, there is a function $f\in V$ so that $L(x) = \langle x , f \rangle$ for all $x\in V^*$. Let $Q$ be in the simplex defined by $P_{n_0},\ldots,P_{n_d}$ where $P_{n_i}\in T$ and $d\leq M-1$.  We need to show that $L(F(Q))>0$.  Recall by the construction of $P$ that for any $p\in X$ that $|f(p)|\leq c(K+\epsilon)$.  Equivalently $|L(E(p))|\leq c(K+\epsilon)$.  Note also that since the $P_{n_i}$ are in $T$, that $L(P_{n_i})=c$.  Now if $Q=\sum x_iP_{n_i}$, $F(Q)$ is a sum of $N$ points of $E(X)$ at least $\sum_i \floor{Nx_i}$ of which are one of the $P_{n_i}$.  Note that $N-\sum_i \floor{Nx_i} = \sum_i \fpart{Nx_i} < \sum_{i=0}^d 1 = d+1 \leq M.$  Therefore, since this term is an integer, $N-\sum_i \floor{Nx_i} \leq M-1$.  Hence $F(Q)$ is a sum of $N-M+1$ of the $P_{n_i}$ (with multiplicity) plus the sum of $M-1$ other points in $E(X)$.  Hence
$$
L(F(Q)) \geq (N-M+1)c - (M-1)(K+\epsilon)c \geq c[N-(M-1)(K+1+\epsilon)] > 0.
$$
This completes our proof.
\end{proof}

To finish the proof of Theorem \ref{PathConnThm} we will use the following:
\begin{prop}\label{TopologicalProp}
Let $Q$ be a polytope in a finite dimensional vector space $U$ with $0$ in the interior of $Q$.  Let $F:Q\rightarrow U$ be a continuous function so that for any facet, $T$, of $Q$ defined by the linear equation $L(x)=c$, with $c>0$, $L(F(T))\subset \R^+$, then $0\in F(Q)$.
\end{prop}
\begin{proof}
We may assume that $Q$ spans $U=\R^n$, since otherwise we may replace $U$ by the span of $Q$ and replace $F$ by its composition with a projection onto this subspace.  Suppose for sake of contradiction that $0\not\in F(Q)$.  Consider the map $f:B^n\rightarrow Q$ defined by letting $f(0)=0$ and otherwise $f(x)=m_x x$ where $m_x$ is the unique positive real number so that $\frac{m_x x}{|x|} \in \partial Q$.  Next consider $g:Q\rightarrow S^{n-1}$ defined by $g(x) = \frac{F(x)}{|F(x)|}$.  Composing we get a map $g\circ f:B^n \rightarrow S^{n-1}$.  Since the map extends to the whole ball, $g\circ f:S^{n-1}\rightarrow S^{n-1}$ must be contractible.  We use our hypothesis on $F$ to show that this map is actually degree 1 and reach a contradiction.

First, we claim that for no $x\in S^{n-1}$ is $g(f(x)) = -x$.  For $x\in S^{n-1}$, $f(x)\in \partial Q$.  Let $f(x)$ land in a facet, $T$, defined by $L(y) = c > 0$.  We have that $L(x) > 0$, because $f(x)$ is a positive multiple of $x$.  We also have that $L(g(f(x))) > 0$ because $g(f(x))$ is a positive multiple of a point in $F(T)$.  Since $L(x)>0$ and $L(g(f(x)))>0$, it cannot be the case that $g(f(x))=-x$.

Finally, we claim that any map $h:S^{n-1}\rightarrow S^{n-1}$ that sends no point to its antipodal point is degree 1.  This is because there is a homotopy from $h$ to the identity by moving each $h(x)$ at a constant rate along the arc from $-x$ to $h(x)$ to $x$.
\end{proof}

Finally, we can prove Theorem \ref{PathConnThm}
\begin{proof}
We construct the polytope $P$ as in Proposition \ref{PConstructionProp} with $\epsilon < \frac{N}{M-1} - K - 1$, and $F$ as in Proposition \ref{FConstructionProp}.  Then by Proposition \ref{TopologicalProp} we have that 0 is in the image of $F$.  Since every point in the image of $F$ is a sum of $N$ points of $E(X)$, we have a design of size $N$ by Lemma \ref{ECriterionLem}.
\end{proof}

\section{Tightness of the Bound}\label{TightnessSec}

In this Section, we demonstrate that, in the generality in which it is stated, the lower bound for $N$ in Theorem \ref{PathConnThm} is tight.  First, we note that although it is possible that $K$ is infinite, this can be indicative of the non-existence of designs of any size.

\begin{prop}
Let $\alpha\in (0,1)$ be an irrational number.  Consider the topological design problem $$(X,\mu,W)=([0,1],\alpha\cdot \delta(x-1) + (1-\alpha)\cdot \delta(x),\textrm{Polynomials of degree at most 4}).$$  Then there is no unweighted design for this problem of any size.
\end{prop}
\begin{proof}
Note that for $f(x)=x^2(1-x)^2$, $\int_X f(x) d\mu(x) = 0$.  Note that for this $f$, $\sup(f)>0$ and $\inf(f)=0$, so $K=\infty$.  If we have a design $p_1,\ldots,p_N,$ then it must be the case that $\sum_i f(p_i)=0$.  Therefore since $f(x)\geq 0$ for all $x\in X$, this implies that $f(p_i)=0$ for all $i$.  Therefore, $p_i\in \{0,1\}$ for all $i$.  Next consider $g(x)=x$.  $\int_X g(x)d\mu(x) = \alpha$.  Therefore, we must have that $\frac{1}{N}\sum g(p_i) = \alpha$.  But for each $i$, we must have $g(p_i)$ is either 0 or 1.  Therefore, this sum is a rational number and cannot be $\alpha$, which is irrational.
\end{proof}

We show that even when $K$ is finite, that a path-connected topological design problem may require that its designs be nearly the size mentioned in Theorem \ref{PathConnThm}.  In particular, we show:

\begin{prop}
Let $m>1$ be an integer and $k\geq 1$, $\epsilon>0$ real numbers.  Then there exists a path-connected topological design problem with $M=m$ and $K\leq k + \epsilon$ that admits no design of size $(m-1)(k+1)$ or less.
\end{prop}
\begin{proof}
First note that by increasing the value of $k$ by $\epsilon/2$ and decreasing $\epsilon$ by a factor of 2, it suffices to construct such a design problem that admits no design of size strictly less than $(m-1)(k+1).$  We construct such a design problem as follows.

Let $X=[0,1]$ and let $\mu$ be the Lebesgue measure.  Let $F:X\rightarrow \R$ be a continuous function with the following properties:
\begin{itemize}
\item $F(x)=k$ for $x\in [0,1/(2k)]$
\item $F(x)=-1$ for $x\in [1/2,1]$
\item $F(x)\in [-1,k]$ for $x\in X$
\item $\int_X F(x)d\mu(x)=0$
\end{itemize}
Notice that such $F$ are not difficult to construct.  Next pick $\delta>0$ a sufficiently small real number (we will discuss how small later).  Let $\phi_i$ for $1\leq i \leq m-1$ be continuous real-valued function on $X$ so that
\begin{itemize}
\item $\phi_i(x)\geq 0$ for all $x$
\item $\supp(\phi_i) \subset [0,1/(4k)]$
\item The supports of $\phi_i$ and $\phi_j$ are disjoint for $i\neq j$
\item $\sup(\phi_i)=1$
\item $\int_X \phi_i(x)d\mu(x)=2\delta$
\end{itemize}
It is not hard to see that this is possible to arrange as long as $\delta$ is sufficiently small.  Let
$$
f_i(x) = \delta - \phi_i(x) + \phi_i(2(1-x)).
$$
It is easy to see that $\int_X f_i(x)d\mu(x)=0$.  We let $W$ be the span of $F$ and the $f_i$.

Since all elements of $W$ already have 0 integral, we have that $V=W$ so $M=\dim(W)$.  The $F$ and the $f_i$ are clearly linearly independent, and hence $M=m$.

We now need to bound $K$.  Consider an element of $V$ of the form $G=aF+\sum a_i f_i$.  It is easy to see that $G$'s values on $[1/(2k),1-1/(4k)]$ are sandwiched between its values on the rest of $X$.  Hence $G$ attains its sup and inf on $[0,1/(2k)]\cup [1-1/(4k),1]$.  Let $s = \sum_i a_i$.  We then have that $G(x) = ak+s\delta -\sum a_i\phi_i(x)$ on $[0,1/(2k)]$ and $G(x) = -a + s\delta + \sum a_i\phi_i(2(1-x))$ on $[1/2,1]$.  Therefore,
$$
\sup(G) = \max(ak+s\delta-\min(a_i,0),-a+s\delta+\max(a_i,0)),
$$
$$
\inf(G) = \min(ak+s\delta-\max(a_i,0),-a+s\delta+\min(a_i,0)).
$$
Suppose for sake of contradiction that $\frac{\sup(G)}{|\inf(G)|}>k+\epsilon$.  This means that $\sup(G)+(k+\epsilon)\inf(G)>0$.  If $\sup(G) = ak+s\delta-\min(a_i,0)$ this is at most
\begin{align*}
&ak+s\delta-\min(a_i,0) + (k+(k/(k+1))\epsilon)(-a + s\delta + \min(a_i,0)) \\ & + \epsilon/(k+1)(ak+s\delta-\max(a_i,0))\\   \leq & (k+1+\epsilon)s\delta - \epsilon/(k+1)\max(a_i,0)\\
 \leq & (k+1+\epsilon)(m-1)\max(a_i,0)\delta - \epsilon/(k+1)\max(a_i,0),
\end{align*}
which is non-positive for $\delta$ sufficiently small.

If on the other hand, $\sup(G) = -a+s\delta+\max(a_i,0)$, then $\sup(G)+(k+\epsilon)\inf(G)$ is at most
\begin{align*}
&-a+s\delta+\max(a_i,0) + (1+\epsilon/(k+1))(ak+s\delta-\max(a_i,0))\\ & + (k-1+k\epsilon/(k+1))(-a+s\delta+\min(a_i,0))\\
\leq &(k+1+\epsilon)s\delta - \epsilon\max(a_i,0)/(k+1)\\
\leq & (k+1+\epsilon)(m-1)\max(a_i) - \epsilon\max(a_i,0)/(k+1)
\end{align*}
which is non-positive for $\delta$ sufficiently small, yielding a contradiction.

Hence, if we picked $\delta$ sufficiently small $\frac{\sup(G)}{|\inf(G)|}\leq k+\epsilon$ for all $G\in V$, so $K\leq k+\epsilon$.

Next suppose that we have a design $x_1,\ldots,x_N$ for this design problem.  Since $\sum f_j(x_i) = 0$ and since $f_j$ is negative only on the support of $\phi_j$, we must have at least $m-1$ of the $x_i$ each in a support of one of the $\phi_j$, and hence there must be at least $m-1$ $x_i$ in $[0,1/(2k)]$.  Next we note that we must also have $\sum F(x_i)=0$.  At least $m-1$ of these $x_i$ are in $[0,1/(2k)]$ and therefore $F$ of these $x_i$ equals $k$.  Therefore since $F(x_j)\geq -1$ for each other $j$, there must be at least $k(m-1)$ other points in our design.  Hence $N$ must be at least $k(m-1)+(m-1) = (m-1)(k+1).$
\end{proof}

\section{The Bound for Homogeneous Spaces}\label{HomogeneousSec}

In this Section, we show that there is a much nicer bound on the size of designs if we have a homogenous, path-connected, topological design problem.
\begin{thm}\label{HomThm}
Let $(X,\mu,W)$ be a homogeneous topological design problem with $M>1$.  Then for any $N > M(M-1)$, there exists a design for $X$ of size $N$.  Furthermore, there exists a design for $X$ of size at most $M(M-1)$.
\end{thm}

We will show that $K\leq (M-1)$, where the equality is strict unless $X$ has a design of size $M$.  An application of Theorem \ref{PathConnThm} then yields our result.

We begin with a Lemma
\begin{lem}\label{WeightedDesignKLem}
If $X$ is a homogenous topological design problem, and if $p_i,w_i$ is a weighted design for $X$, then $K \leq \frac{1-\max(w_i)}{\max(w_i)}$.
\end{lem}
\begin{proof}
Without loss of generality, $w_1=\max(w_i)$.  Suppose for sake of contradiction that $K>\frac{1-w_1}{w_1}$.  This means that there is an $f\in V$ so that $\frac{\sup(f)}{|\inf(f)|} > \frac{1-w_1}{w_1}$.  This means that there is a $p\in X$ so that $w_1 f(p) + (1-w_1) \inf(f) > 0$.  Since $X$ is homogenous, there is a $g:X\rightarrow X$ preserving all properties of the design problem so that $g(p_1)=p$.  Since $g$ preserves $\mu$ and $W$, $g(p_i),w_i$ must also be a weighted design for $X$.  Therefore, $\sum_i w_i f(g(p_i)) = 0$.  But on the other hand this is
$$
w_1f(p) + \sum_{i>1} w_i f(g(p_i)) \geq w_1 f(p) + (1-w_1)\inf(f) > 0,
$$
yielding a contradiction.
\end{proof}

We note the following interesting pair of Corollaries.
\begin{cor}
If $X$ is a homogeneous topological design problem, and $p_i,w_i$ a weighted design for $X$, then $\max(w_i)\leq \frac{1}{K+1}$.
\end{cor}

\begin{cor}\label{HomLowerBoundCor}
If $X$ is a homogeneous topological design problem, $X$ admits no weighted design of size less than $K+1$.
\end{cor}

We will also need one more Lemma
\begin{lem}\label{SmallDesignLem}
If $X$ is a path-connected topological design problem and $M>0$, $X$ has a weighted design of size at most $M$.
\end{lem}
\begin{proof}
Suppose for sake of contradiction that there is no such weighted design.  Then it must be the case that there are no $p_i\in X$ and $w_i\geq 0$ for $1\leq i \leq M$ so that $\sum_i w_i E(p_i)=0$.  This means that whenever a non-negative linear combination of $M+1$ values of $E(p_i)$ equals 0, the weights must be all 0 or all positive.  By Lemma \ref{WeightedDesignLem} there must be some $M+1$ points for which some non-negative linear combination equals 0.  As we deform our set of points, it will always be the case that some linear combination equals 0 by a dimension count.  Furthermore, the coefficients of this combination will vary continuously.  Since, by assumption, it is never possible to write 0 as a non-negative linear combination with at least one coefficient equal to 0, it must be the case that no matter how we deform the $p_i$, there will always exist a linear combination equal to 0 with strictly positive coefficients.  But this is clearly not the case if all of the $p_i$ are equal to some point $p$ on which not all of the functions in $V$ vanish.
\end{proof}

We can now prove Theorem \ref{HomThm}.
\begin{proof}
By Lemma \ref{SmallDesignLem}, there is a weighted design for $X$ of size at most $M$.  If all of the weights are equal, this is a design of size $M$, and by Lemma \ref{WeightedDesignKLem} $K\leq \frac{1-1/M}{1/M} = M-1$ and the remainder of the result follows from Theorem \ref{PathConnThm}.  If the weights of this design are not equal, some weight is larger than $\frac{1}{M}$, and hence $K<\frac{1-1/M}{1/M}=M-1$, and again our result follows from Theorem \ref{PathConnThm}.
\end{proof}

\subsection{Examples}\label{CorrSec}

We provide several Corollaries of Theorem \ref{HomThm}.

\begin{cor}
There exists a spherical design of strength $n$ on the $d$-dimensional sphere of size $O(n^{2d}/(d!^2))$.
\end{cor}

\begin{cor}
There exists a design of strength $n$ on the Grassmannian, $G(m,k)$ of size $O_{m,k}(n^{2k(m-k)})$.
\end{cor}

\subsection{Conjecture}

Although we prove a bound of size $O(M^2)$ for homogeneous path-connected topological design problems, it feels like the correct result should be $O(M)$, since that is roughly the number of degrees of freedom that you would need.  We can rephrase the problem for homogeneous path-connected spaces a little though.

First, we may replace $X$ by $E(X)$, which is a bounded subset of $V^*$.  Next, we note that the $L^2$ measure on $V$ is preserved by the symmetries of $X$.  Hence the symmetry group $G$ of $X$ (which is transitive by assumption) is a subgroup of $O(V^*)$, and hence compact.  Since $X$ is a quotient of the identity component $G_0$ of $G$ we may pull our design problem back to one on $G_0$ (using the pullbacks of $\mu$ and $W$).  Since $G_0$ is also a path-connected subgroup of $O(V^*)$, it must be a Lie group.  Hence we have reduced the problem of finding a design in a path-connected homogenous topological design problem to finding one in a design problem of the following form:

$X=G$ is a compact Lie Group.  $\mu$ is the normalized Haar measure for $G$.  $W$ is a left-invariant, finite dimensional space of functions on $G$.  Since $L^2(G)$ decomposes as a sum $\bigoplus_{\rho_i\in\hat{G}} \phi_i \otimes \phi_i^*$, $W$ must be a sum of the form $\bigoplus_{\rho_i\in\hat{G}} \rho_i \otimes W_i$ where $W_i$ is a subspace of $\rho_i^*$ and all but finitely many $W_i$ are 0.

Note that although we have all this structure to work with, proving better bounds even for the circle seems to be non-trivial.  This Conjecture says in that case that given any $M$ distinct non-zero integers $n_i$ that there should exist $O(M)$ complex numbers $z_j$ with $|z_j|=1$ so that $\sum_j z_j^{n_i} = 0$ for all $i$.

\section{Designs on the Interval}\label{IntervalSec}

Let $I$ be the interval $[-1,1]$.  For $\alpha, \beta \geq -\frac{1}{2}$ let $\mu_{\alpha,\beta}$ be the measure $\frac{(1-x)^\alpha (1+x)^\beta\Gamma(\alpha+\beta+2)}{2^{\alpha+\beta+1}\Gamma(\alpha+1)\Gamma(\beta+1)} dx$ on $I$.  Let $\mathcal{P}_n$ be space of polynomials of degree at most $n$ on $I$.  We will prove the following Theorem:

\begin{thm}\label{IntervalDesignThm}
The size of the smallest design for $(X,\mu_{\alpha,\beta},\mathcal{P}_n)$ is $\Theta_{\alpha,\beta}(n^{2\max(\alpha,\beta)+2}).$
\end{thm}

Where above and throughout the paper, $O_a(N)$ denotes a quantity bounded above by $N$ times some absolute constant depending only on $a$, and $\Theta_a(N)$ denotes a quantity bounded above and below by positive multiples of $N$ that depend only on $a$.

Several others have considered the problem of finding designs for this design problem.  Bernstein proved in \cite{abone} the existence of such designs of size $O(n^2)$ for $\alpha=\beta=0$.  This work was latter extended by Kuijlaars, who proved asymptotically optimal upper bounds for $\alpha=\beta\geq 0$ in \cite{abeq} and for $\alpha,\beta\geq 0$ in \cite{abpos}.  Theorem \ref{IntervalDesignThm} extends these results to the case of $\alpha$ and $\beta$ negative.

In order to prove this Theorem, we will first need to review some basic facts about Jacobi polynomials.  We will use \cite{kn:sze} as a guide.

\begin{defn}
We define the Jacobi polynomials inductively as follows:
For $n$ a non-negative integer and $\alpha,\beta\geq -\frac{1}{2}$,
$P_n^{(\alpha,\beta)}(x)$ is the unique degree $n$ polynomial with
$$
P_n^{(\alpha,\beta)}(1) = \binom{n+\alpha}{n}
$$
and so that $\jac{n}$ is orthogonal to $\jac{k}$ for $k<n$ with respect to the inner product $\langle f,g\rangle = \int_I f(x)g(x) d\mu_{\alpha,\beta}(x)$.
\end{defn}

Hence the $\jac{n}$ are a set of orthogonal polynomials for the measure $\mu_{\alpha,\beta}$.  The normalization is given by \cite{kn:sze} Equation (4.3.3)
\begin{align}\label{JacobiNormalizationEquation}
\int_I (\jac{n})^2 d\mu_{\alpha,\beta} & = \frac{\Gamma(n+\alpha+1)\Gamma(n+\beta+1)\Gamma(\alpha+\beta+2)}{(2n+\alpha+\beta+1)\Gamma(n+1)\Gamma(n+\alpha+\beta+1)\Gamma(\alpha+1)\Gamma(\beta+1)}\nonumber \\ & = \Theta_{\alpha,\beta}(n^{-1}).
\end{align}
Hence we define the normalized orthogonal polynomials
\begin{align*}
R_n^{(\alpha,\beta)} & = \jac{n}\sqrt{\frac{(2n+\alpha+\beta+1)\Gamma(n+1)\Gamma(n+\alpha+\beta+1)\Gamma(\alpha+1)\Gamma(\beta+1)}{\Gamma(n+\alpha+1)\Gamma(n+\beta+1)\Gamma(\alpha+\beta+2)}}\\ & = \jac{n}\Theta_{\alpha,\beta}(\sqrt{n}).
\end{align*}

We will also need some more precise results on the size of these polynomials.  In particular we have Theorem 8.21.12 of \cite{kn:sze} which states that
\begin{eqnarray}\label{JacobiApproximationEqn}
\left( \sin \frac{\theta}{2}\right)^\alpha\left( \cos \frac{\theta}{2}\right)^\beta\jac{n}(\cos\theta) = \frac{N^{-\alpha}\Gamma(n+\alpha+1)}{n!}\sqrt{\frac{\theta}{\sin \theta}} J_\alpha(N\theta)\nonumber \\
+
\begin{cases}
\theta^{1/2}O(n^{-3/2}) \ & \textrm{if} \ cn^{-1} \leq \theta \leq \pi - \epsilon \\
\theta^{\alpha+2}O(n^\alpha) \ & \textrm{if} \ 0 < \theta \leq cn^{-1}
\end{cases}
\end{eqnarray}
for any positive constants $c$ and $\epsilon$ and where $N= n +(\alpha+\beta+1)/2$, and $J_\alpha$ is the Bessel function.

We will also want some bounds on the size of the Bessel functions.  From \cite{kn:sze} (1.71.10) and (1.71.11) we have that for $\alpha \geq -\frac{1}{2}$
$$
J_\alpha(x) \sim c_\alpha(x^\alpha) \ \textrm{as} \ x\rightarrow 0
$$
and
$$
J_\alpha(x) = O_\alpha(x^{-1/2}).
$$

The first of these along with Equation \ref{JacobiApproximationEqn} implies that $\jac{n}$ has no roots within $O_{\alpha,\beta}(n^{-2})$ of 1.  Noting that $\jac{n}(x)$ is a constant multiple of $P_n^{(\beta,\alpha)}(-x)$, it also has no roots within $O_{\alpha,\beta}(n^{-2})$ of -1.  Applying Theorem (8.21.13) of \cite{kn:sze}, we also find that $\jac{n}$ has roots within $O_{\alpha,\beta}(n^{-2})$ of either endpoint.  Applying Equation \ref{JacobiApproximationEqn}, we find that for $x\in I$
\begin{equation}\label{RBoundEqn}
R_n^{(\alpha,\beta)}(x) = O_{\alpha,\beta}\left((1-x)^{-\alpha/2-1/4} (1+x)^{-\beta/2-1/4} \right).
\end{equation}

We will need to make use of Gauss-Jacobi quadrature which, for completeness, we state here.

\begin{lem}\label{IntervalWeightedDesignLem}
Let $\mu$ be a normalized measure on $I$.  Let $R_n^\mu$ be the sequence of orthogonal polynomials for $\mu$.  (i.e. $R_n^\mu$ is a polynomial of degree $n$, and $\{R_0^\mu,R_1^\mu,\ldots,R_n^\mu\}$ is an orthonormal basis for $\mathcal{P}_n$ with the inner product $\langle f,g\rangle_\mu = \int_I f(x)g(x)d\mu(x).$)  Let $r_i$ be the roots of $R_n^\mu(x)$.  Let $w_i = \frac{1}{\sum_{j=0}^{n-1} (R_j^\mu(r_i))^2}$.  Then $(w_i,r_i)$ is a weighted design for $(I,\mu,\mathcal{P}_{2n-1})$.
\end{lem}

We are now prepared to show that all designs for $(I,\mu_{\alpha,\beta},\mathcal{P}_n)$ are reasonably large.

\begin{prop}\label{IntervalDesignLowerBoundProp}
If $\alpha,\beta\geq -\frac{1}{2}$, then all unweighted designs for $(I,\mu_{\alpha,\beta},\mathcal{P}_n)$ have size $\Omega_{\alpha,\beta}(n^{2\alpha+2})$.
\end{prop}
\begin{proof}
We increase $n$ by a factor of 2, and instead prove bounds on the size of designs for $(I,\mu_{\alpha,\beta},\mathcal{P}_{2n})$.

Let $r_n$ be the biggest root of $R_n^{(\alpha,\beta)}$.  Since $p(x)=\frac{\left(R_n^{(\alpha,\beta)}(x)\right)^2}{(x-r_n)}$ is $R_n^{(\alpha,\beta)}(x)$ times a polynomial of degree less than $n$, $\int_I p d\mu_{\alpha,\beta} = 0$.  Since $p(x)$ is positive outside of $[r_n,1]$, any design must have a point in this interval.  Therefore any design must have at least one point in $[1-O_{\alpha,\beta}(n^{-2}),1]$.  If such a point is written as $\cos\theta$ then $\theta = O_{\alpha,\beta}(n^{-1})$.  For $c$ a sufficiently small constant (depending on $\alpha$ and $\beta$), define
$$
f(x) = \frac{\left(\sum_{i=cn}^{2cn} \njac{i}(x)\right)^2}{cn}.
$$
It is clear that $f(x)\geq 0$ for all $x$, and clear from the orthonormality that $\int_I fd\mu_{\alpha,\beta} = 1$.  On the other hand, for $c$ sufficiently small and $cn\leq i \leq 2cn$, Equation \ref{JacobiApproximationEqn} tells us that
$$
\njac{i}(x) = \Omega_{\alpha,\beta}(n^{\alpha+1/2})
$$
on $[1-r_n,1]$.  Therefore
$$f(x) = \Omega_{\alpha,\beta}(n^{2\alpha+2})$$
on $[1-r_n,1]$.  Therefore if $p_1,\ldots,p_N$ is a design for $(I,\mu_{\alpha,\beta},\mathcal{P}_n)$, we may assume that $p_1\in [1-r_n,1]$ and we have that
\begin{align*}
1 & = \frac{1}{N} \sum_{i=1}^N f(p_i)\\
& \geq \frac{f(p_1)}{N}\\
& \geq \Omega_{\alpha,\beta}(n^{2\alpha+2}N^{-1}).
\end{align*}
Therefore $N=\Omega(n^{2\alpha+2})$.
\end{proof}

In order to prove the upper bound, we use a slightly more sophisticated version of our previous techniques.  First, we need to define some terminology.

\begin{defn}
Let $f:[0,1]\rightarrow \R$ we define $\var(f)$ to be the total variation of $f$ on $[0,1]$.  For $\gamma:[0,1]\rightarrow X$ and $f:X\rightarrow \R$, we define $\var_\gamma(f) = \var(f \circ \gamma)$.
\end{defn}

\begin{defn}
For a design problem $(X,\mu,W)$ and a map $\gamma:[0,1]\rightarrow X$ we define
$$
K_\gamma = \sup_{f\in V\backslash \{0\}}\left( \frac{\var_\gamma(f)}{\max\left(\sup_{\gamma([0,1])}(f),0\right)}\right).
$$
\end{defn}

It should be noted that as a consequence of this definition that if there are $f\in V\backslash\{0\}$ that are non-positive on $\gamma([0,1])$ that this will cause $K_\gamma$ to be infinite.  It should be noted that in such cases, it will usually not be the case that there will be any design supported only on the image of $\gamma$. If no such $f$ exists, a compactness argument shows that $K_\gamma$ is finite.

We note that replacing $f$ by $g=\frac{\sup_{\gamma([0,1])}(f)-f}{\sup_{\gamma([0,1])}(f)}$, we have that $g\geq 0$ on $\gamma([0,1])$, $\int_X g=1$, and $\var_\gamma(g)=\frac{\var_\gamma(f)}{\sup_{\gamma([0,1])}(f)}$.  Hence we have the alternative definition
$$
K_\gamma = \sup_{\substack{g\in W\oplus 1\\ g\geq 0 \ \textrm{on} \ \gamma([0,1])\\ \int_X gd\mu = 1}} \var_\gamma(g).
$$
Or equivalently, scaling $g$ by an arbitrary positive constant,
$$
K_\gamma = \sup_{\substack{g\in W\oplus 1\\ g\geq 0 \ \textrm{on} \ \gamma([0,1])}} \frac{\var_\gamma(g)}{\int_X gd\mu}.
$$

\begin{prop}\label{IntervalConstructionProp}
Let $(X,\mu,W)$ be a topological design problem with $M>0$.  Let $\gamma:[0,1]\rightarrow X$ be a continuous function with $K_\gamma$ finite.  Then for any integer $N>K_\gamma/2$ there exists a design for $(X,\mu,W)$ of size $N$.
\end{prop}
\begin{proof}
Let $\frac{2N}{K_\gamma}-1>\epsilon>0$.  For every $f\in V\backslash\{0\}$, there exists an $x\in[0,1]$ so that $K_\gamma f(\gamma(x))(1+\epsilon) >  \var_\gamma(f)$.  Since this property is invariant under scaling of $f$ by positive real numbers, and since it must also hold for some open neighborhood of $f$, by compactness, we may pick finitely many $x_i$ so that for any $f\in V\backslash\{0\}$,
$$K_\gamma \max_if(\gamma(x_i)) > (1-\epsilon) \var_\gamma(f).$$

Let $P$ be the polytope in $V^*$ spanned by the points $E(\gamma(x_i))$.  We will define a function $F:P\rightarrow V^*$ with the following properties:
\begin{itemize}
\item $F$ is continuous
\item For each $x\in P$, $F(x)$ can be written as $\sum_{i=1}^N E(\gamma(y_i))$ for some $y_i\in [0,1]$
\item For each facet $T$ of $P$ defined by $L(x) = c > 0$, $L(F(T))\subset \R^+$
\end{itemize}
Once we construct such an $F$, we will be done by Proposition \ref{TopologicalProp}.

Suppose that our set of $x_i$ is $x_1<x_2<\ldots<x_R$.  We first define a continuous function $C:P\rightarrow \R^R$ whose image consists of points with non-negative coordinates that add to 1.  This is defined as follows.  First, we triangulate $P$.  Then for $y\in P$ in the simplex spanned by, say, $\{E(\gamma(x_{i_1})),E(\gamma(x_{i_2})),\ldots,E(\gamma(x_{i_k}))\}$.  We can then write $y$ uniquely as $\sum_{j=1}^k w_j E(\gamma(x_{i_j}))$ for $w_j\geq 0$ and $\sum_j w_j=1$.  We then define $C(y)$ to be $w_j$ on its $i_j$ coordinate for $1\leq j \leq k$, and 0 on all other coordinates.  This map is clearly continuous within a simplex and its definitions on two simplices agree on their intersection.  Therefore, $C$ is continuous.

For $w\in \R^R$ with $w_i\geq 0$ and $\sum_i w_i=1$, we call $w_i$ a set of weights for the $x_i$.  Given such a set of weights define $u_w:[0,1]\rightarrow [0,N+1]$ to be the increasing, upper semi-continuous function
$$
u_w(x) = x + N \sum_{x_i\leq x} w_i.
$$
For integers $1\leq i \leq N$ define
$$
p_i(w) = \inf \{ x: u_w(x) \geq i \}.
$$
Note that $p_i(w)$ is continuous in $w$.  This is because if $|w-w'|<\delta$ (in the $L^1$ norm) then $|u_w(x)-u_{w'}(x)|<N\delta$ for all $x$.  Therefore, since $u_{w'}(x+N\delta) \geq u_{w'}(x) + N\delta$ we have that $|p_i(w)-p_i(w')|\leq N\delta$.  We now define $F$ by
$$
F(y) = \sum_{i=1}^N E(\gamma(p_i(C(y)))).
$$

This function clearly satisfies the first two of our properties, we need now to verify the third.  Suppose that we have a face of $P$ defined by the equation $\langle f, y \rangle = 1$ for some $f\in V$.  We then have that $\sup_i (f(\gamma(x_i))) = 1$.  Therefore $\var_\gamma(f) < K_\gamma (1+\epsilon)$.  Let this face of $P$ be spanned by $E(\gamma(x_{i_1})),\ldots,E(\gamma(x_{i_M}))$ for $i_1<i_2<\ldots<i_M$.  It is then the case that $f(\gamma(x_{i_j}))=1$ for each $j$.  Letting $w=C(y)$, it is also the case that $w_k$ is 0 unless $k$ is one of the $i_j$.

Note that $\lim_{x\rightarrow x_{i_1}^-}u_w(x) <1$ and $u(x_{i_M})>N$.  This implies that none of the $p_i(w)$ are in $[0,x_{i_1})$ or $(x_{i_M},1]$.  Additionally, note that
$$
\lim_{x\rightarrow x_{i_{n+1}}^-}u(x) - u(x_{i_n}) = x_{i_{n+1}}-x_{i_n} < 1.
$$
This implies that there is at most one $p_i$ in $(x_{i_n},x_{i_{n+1}})$ for each $n$.  For a point $x$ in this interval we have that $|f(\gamma(x))-1|$ is at most half of the total variation of $f\circ \gamma$ on $[x_{i_n},x_{i_{n+1}}]$.  All other $p_i(w)$ must be one of the $x_{i_j}$.  Therefore summing over all $p_i(w)$, we get that
$$
|N-f(F(y))| = \left| N - \sum_{i=1}^N f(\gamma(p_i(w))) \right| \leq \sum_{i=1}^N |1-f(\gamma(p_i(w)))|
$$
which is at most half of the variation of $f\circ \gamma$ on $[x_{i_1},x_{i_M}].$  This in turn is at most $\frac{K_\gamma (1+\epsilon)}{2} < N$.  Therefore $f(F(y))>0$.  This proves that $F$ has the last of the required properties and completes our proof.
\end{proof}

In order to prove the upper bound for Theorem \ref{IntervalDesignThm}, we will apply this proposition to $\gamma:[0,1]\rightarrow I$ defined by $\gamma(x)=2x-1$.  We begin with the case of $\alpha=\beta=-\frac{1}{2}$.

\begin{lem}\label{circleVariationLem}
For $(I,\mu_{-1/2,-1/2},\mathcal{P}_n)$ and $\gamma$ as described above, $K_\gamma = O(n)$.
\end{lem}
\begin{proof}
We will use the alternative definition of $K_\gamma$, namely the $\sup$ over $f\in W\bigoplus 1$, non-negative on $\gamma([0,1])$ and $\int fd\mu_{-1/2,-1/2}=1$, of $\var_\gamma(f)$.  If $f\geq 0$ on $\gamma([0,1]) = I$, then $f$ must be a sum of squares of polynomials of degree at most $n/2+1$ plus $(1-x^2)$ times a sum of such polynomials.  Since $\int_I fd\mu$ is linear and $\var_\gamma(f)$ sublinear, it suffices to check for $f=g^2$ or $f=(1-x^2)g^2$.  Note that $\mu_{-1/2,-1/2}$ is the projected measure from the circle to the interval.  Therefore, we can pull $f$ back to a function on the circle either of the form $g(\cos\theta)^2$ or $(\sin\theta g(\cos\theta))^2$.  In either case, $\int_{S^1} f(\theta)d\theta=1$ and $f(\theta)=h(\theta)^2$ for some polynomial $h$ of degree $O(n)$.  It suffices to bound the variation of $f$ on the circle.  In particular it suffices to show that $\int_{S^1} |f'(\theta)|d\theta = O(n)$.

We note that $$|h|_2^2 = \int_{S^1} h^2(\theta)d\theta = \int_{S^1} f(\theta)d\theta = 1.$$  We also note that
$$
\int_{S^1} |f'(\theta)|d\theta = 2\int_{S^1} |h(\theta)h'(\theta)|d\theta \leq 2|h|_2|h'|_2.
$$
Hence it suffices to prove that for $h$ a polynomial of degree $m$ that $|h'|_2 = O(m)|h|_2.$  This follows immediately after noting that the orthogonal polynomials $e^{ik\theta}$ diagonalize the derivative operator.
\end{proof}

We now relate this to functions for arbitrary $\alpha$ and $\beta$.

\begin{lem}\label{AlphaComparisonLem}
Let $\alpha,\beta\geq -\frac{1}{2}$.  Let $f\in \mathcal{P}_n$, $f\geq 0$ on $I$.  Then
$$
\int_I fd\mu_{\alpha,\beta} = \Omega_{\alpha,\beta}(n^{-2\max(\alpha,\beta)-1})\int_I fd\mu_{-1/2,-1/2}.
$$
\end{lem}
\begin{proof}
We rescale $f$ so that $\int_I f d\mu_{-1/2,-1/2}=1.$ We let $r_i$ be the roots of $P_{n+1}^{(-1/2,-1/2)}$.  By Lemma \ref{IntervalWeightedDesignLem}, there are weights $w_i$ making this a design for $(I,\mu_{-1/2,-1/2},\mathcal{P}_{2n})$.  By Equation \ref{RBoundEqn}, we have that
$$
w_i = \Omega(n^{-1}).
$$

We have that $\sum_i w_i f(r_i)=1$.  Therefore, since $f(r_i)\geq 0$, we have that
$$\sum_i w_i f(r_i)^2 \leq \frac{1}{\min w_i} = O(n).$$
Hence $\int_I f^2d\mu_{-1/2,-1/2} = O(n)$.  Let $R\subset I$ be $R=[1-cn^{-2},1]\cup[-1,-1+cn^{-2}]$ for $c$ a sufficiently small positive constant. Let $I_R$ be the indicator function of the set $R$.  Then
\begin{align*}
\int_I I_R^2d\mu_{-1/2,-1/2} & = \int_R  d\mu_{-1/2,-1/2}\\
& = O\left(\int_{1-cn^{-2}} (1-x)^{-1/2}dx \right)\\
& = O(\sqrt{c}n^{-1}).
\end{align*}
Therefore
$$
\int_R f d\mu_{-1/2,-1/2} = \int_I I_R f d\mu_{-1/2,-1/2} \leq |f|_2|I_R|_2 = O(\sqrt{c}).
$$
Hence for $c$ sufficiently small, $\int_R f d\mu_{-1/2,-1/2} \leq \frac{1}{2}$.  Therefore $\int_{I\backslash R} fd\mu_{-1/2,-1/2} \geq \frac{1}{2}$.  But since the ratio of the measures $\frac{\mu_{\alpha,\beta}}{\mu_{-1/2,-1/2}} = \Omega_{\alpha,\beta}(1-x)^{\alpha+1/2}(1+x)^{\beta+1/2}$ is at least $\Omega_{\alpha,\beta}(n^{-2\max(\alpha,\beta)-1})$ on $I\backslash R$, we have that
\begin{align*}
\int_I fd\mu_{\alpha,\beta} \geq \int_{I\backslash R} f d\mu_{\alpha,\beta} &= \Omega_{\alpha,\beta}(n^{-2\max(\alpha,\beta)-1})\int_{I\backslash R} f d\mu_{-1/2,-1/2} \\ &= \Omega_{\alpha,\beta}(n^{-2\max(\alpha,\beta)-1}).
\end{align*}
\end{proof}

We can now extend Lemma \ref{circleVariationLem} to our other measures

\begin{lem}\label{LineVariationLem}
Consider $(I,\mu_{\alpha,\beta},\mathcal{P}_n)$ and $\gamma$ as above.  Then $K_\gamma=O_{\alpha,\beta}(n^{2\max(\alpha,\beta)+2})$.
\end{lem}
\begin{proof}
We use the alternative description of $K_\gamma$.  Let $f\in \mathcal{P}_n$ with $f\geq 0$ on $I$ and $\int_I fd\mu_{\alpha,\beta} = 1.$  By Lemma \ref{AlphaComparisonLem}, $\int_I fd\mu_{-1/2,-1/2} = O_{\alpha,\beta}(n^{2\max(\alpha,\beta)+1}).$  Therefore using Lemma \ref{circleVariationLem}, $\var_\gamma(f) \leq O(n)O_{\alpha,\beta}(n^{2\max(\alpha,\beta)+1}) = O_{\alpha,\beta}(n^{2\max(\alpha,\beta)+2}).$  Therefore since this holds for all such $f$, $K_\gamma = O_{\alpha,\beta}(n^{2\max(\alpha,\beta)+2}).$
\end{proof}

Theorem \ref{IntervalDesignThm} now follows from Proposition \ref{IntervalDesignLowerBoundProp}, Proposition \ref{IntervalConstructionProp} and Lemma \ref{LineVariationLem}.

\section{Spherical Designs}\label{SphereSec}

In this Section, we will focus on the problem of designs on a sphere.  In particular, for integers $d,n>0$ let $\mathcal{D}^d_n$ denote the design problem given by the $d$-sphere with its standard, normalized measure, and $W$ the space of polynomials of total degree at most $n$.  We begin by proving lower bounds:

\begin{thm}\label{SphericalLowerBoundThm}
Any weighted design for $\sd{d}{n}$ is of size $\Omega_d(n^d)$.
\end{thm}
\begin{proof}
Let $U$ be the space of polynomials of degree at most $n/2$ on $S^d$.  Note that $\dim(U) = \Omega_d(n^d)$.  We claim that $K\geq M':=\dim(U)$.  Pick $x\in S^d$.  Let $\phi_1,\ldots,\phi_{M'}$ be an orthonormal basis of $U$.  Let $f(y)=(\sum_i \phi_i(y)\phi_i(x))^2$.  It is clear that $\int_{S^d} f d\mu = \sum_i \phi_i(x)^2$.  Also $f(x) = (\sum_i \phi_i(x)^2)^2$.  Let $g(y)=\sum_i \phi_i(y)^2$.  $g$ is clearly invariant under the action of $SO(d+1)$, and is therefore constant.  Furthermore, $\int_{S^d} gd\mu = M'$.  Therefore $g(x)=M'$.  Therefore, $\int_{S^d}fd\mu = M'$ and $f(x)=(M')^2$.  Since $f\geq 0$ on $S^d$, $K\geq \frac{f(x)}{\int_{S^d} f d\mu} = M'$.

Therefore since the action of $SO(d)$ makes $\sd{d}{n}$ a homogeneous design problem Corollary \ref{HomLowerBoundCor} implies that any weighted design for $\sd{d}{n}$ must have size at least $M'=\Omega_d(n^d)$.
\end{proof}

We also prove a nearly matching lower bound.  Namely:

\begin{thm}\label{SphereicalDesignThm}
For $N=\Omega_d(n^d \log(n)^{d-1})$, there exists a design for $\sd{d}{n}$ of size $N$.
\end{thm}

The proof of Theorem \ref{SphereicalDesignThm} again uses Proposition \ref{IntervalConstructionProp}, but the choice of $\gamma$ is far less obvious than it is when applied in Theorem \ref{IntervalDesignThm}.  In fact, we will want to introduce a slight generalization of the terminology first.

\begin{defn}
Let $G$ be a topological graph.  If $\gamma:G\rightarrow X$ and $f:X\rightarrow \R$ are functions, define $\var_\gamma(f)$ as follows.  For each edge $e$ of $G$ let $\gamma_e:[0,1]\rightarrow X$ be the map $\gamma$ restricted to $e$.  Then
$$
\var_\gamma(f) := \sum_{e\in E(G)} \var_{\gamma_e}(f).
$$
\end{defn}

Note that for an embedded graph $G$, we will often simply refer to $\var_G(f)$.

\begin{defn}
For $(X,\mu,M)$ a design problem, $G$ a graph, and $\gamma:G\rightarrow X$ a function, define
$$
K_\gamma = \sup_{f\in V\backslash \{0\}}\left( \frac{\var_\gamma(f)}{\max\left(\sup_{\gamma(G)}(f),0\right)}\right).
$$
\end{defn}

Note that we have alternative definitions of $K_\gamma$ in the same way as we did before.  We will often ignore the function $\gamma$ and simply define $K_G$ for $G$ and embedded graph in $X$.  We note the following version of Proposition \ref{IntervalConstructionProp}:

\begin{prop}\label{KGraphProp}
Let $(X,\mu,W)$ be a topological design problem.  Let $G$ be a connected graph and $\gamma:G\rightarrow X$ a continuous function.  If $K_G$ is finite, and $N>K_G$ is an integer, then $(X,\mu,W)$ admits a design of size $N$.
\end{prop}
\begin{proof}
Note that if we double all of the edges of $G$ that the resulting multigraph admits an Eulerian circuit.  This gives us a continuous map $\gamma':[0,1]\rightarrow X$ that covers each edge of $G$ exactly twice.  Therefore for every function $f$, $\sup_G (f) = \sup_{\gamma([0,1])}(f)$ and $\var_{\gamma'}(f) = 2\var_G(f)$.  Hence $K_{\gamma'}=2K_G$, and the result follows from Proposition \ref{IntervalConstructionProp}.
\end{proof}

We will now need to prove the following:

\begin{prop}\label{SphereGraphProp}
For $d,n\geq 1$, there exists a connected graph $G$ for the design problem $\sd{d}{n}$ so that $K_G = O_d(n^d \log(n)^{d-1})$.  Furthermore this can be done is such a way that the total length of all the edges of $G$ is $n^{O_d(1)}$.
\end{prop}

The basic idea of the proof of Proposition \ref{SphereGraphProp} is as follows.  First, by projecting $S^d$ down onto its first $d-1$ coordinates, we can think of it as a circle bundle over $B^{d-1}$.  We construct our graphs by induction on $d$.  We pick a number of radii $r_i$, and place our graphs for various strength designs on the spheres of radius $r_i$ in $B^{d-1}$.  We also add the loops over the points on these graphs given by the corresponding designs.  The first step is to show that average value of $f$ over our loops in $G$ is roughly the average value over the sphere (see Lemma \ref{approximateDesignLem}).  Naively, this should hold since the average value of $f$ on the sphere of radius $r_i$ in $B^{d-1}$ should equal the average value of $f$ over the appropriate loops (because the loops are arranged in a design).  Our radii will themselves by arranged in an appropriate design, so that the value of $f$ on the sphere will equal the average of the values at there radii.  Unfortunately, our component designs will be of insufficient strength for this to hold.  This is fixed by showing that the component of $f$ corresponding to high degree spherical harmonics at small radius $r_i$ in $B^{d-1}$ is small (this is shown in Lemma \ref{L2toMaxLem}).  The bound on $K_G$ comes from noting that the variation of $f$ along $G$ is given by the sum of variations on the subgraphs.  These in turn are bounded by the size of $f$ on these subgraphs, and the appropriate sum of variations is bounded by the size of $f$ on the whole sphere.

Before we proceed, we will need the following technical results:

\begin{lem}\label{MaxAndDerivativeBounds}
Let $f\in \mathcal{P}_n$.  Then $\sup_{S^{d}} (f) = O(n^{d/2})|f|_2$, $\sup_{S^{d}} {|f'|} = O(n^{d/2+1})|f|_2.$
\end{lem}
\begin{proof}
Let $\phi_i$ $(1\leq i \leq M)$ be an orthonormal basis of the polynomials of degree at most $n$ on $S^{d}$, so that each of the $\phi_i$ is a spherical harmonic.  Note that $M=O(n^d)$. Write $f(u)=\sum a_i \phi_i(u)$.  For $v\in S^d$, $f(v) = \sum a_i \phi_i(v) \leq \sqrt{\sum_i a_i^2} \sqrt{\sum_i \phi_i(v)^2} = |f|_2\sqrt{\sum_i \phi_i(v)^2}$.  Now by symmetry, $\sum_i \phi_i(u)^2$ is a constant function of $u$.  Since it's average value is $M$, $\sum_i \phi_i(v)^2=M$.  Therefore, $f(v) \leq \sqrt{M}|f|_2$.

We also have that
\begin{align*}
|f'(v)| & \leq \sum_i a_i |\phi_i'(v)|  \leq \sqrt{\sum_i a_i^2}\sqrt{\sum_i |\phi_i'(v)|^2}  = |f|_2\sqrt{\sum_i |\phi_i'(v)|^2}.
\end{align*}
Now $\sum_i |\phi_i'(u)|^2$ is a constant function of $u$.  Its average value is
\begin{align*}
\int \sum_i |\phi_i'(u)|^2du & = \sum_i \int |\phi_i'(u)|^2du\\
& = \sum_i \int \phi_i(u)\triangle\phi_i(u) du.
\end{align*}
So $\triangle\phi_i(u) = k^2 \phi_i(u)$ for some $k\leq n$.  Therefore, this is at most $n^2 M$.  Hence,
$|f'(v)| = O(n^{d/2+1})|f|_2.$
\end{proof}

\begin{lem}\label{L2toMaxLem}
For $n\geq d,k\geq 1$ integers, and $f$ a polynomial of degree at most $n$ on the $d$-disk, $D$, with $\int_D f^2(r) (1-r^2)^{(k-2)/2}\frac{dr}{\vol(D)}=1$ then $\sup_{D} f = O\left(\sqrt{\frac{2}{d\beta(d/2,k/2)}}\right)O\left(\frac{n}{d+k-1}\right)^{(d+k-1)/2}.$
\end{lem}
\begin{proof}
Notice that
\begin{align*}
\int_D (1-r^2)^{(k-2)/2}\frac{dr}{\vol(D)} & = d \int_0^1 r^{d-1} (1-r^2)^{(k-2)/2}dr\\
& = d/2 \int_0^1 s^{(d-2)/2}(1-s)^{(k-2)/2}ds\\
& = d\beta(d/2,k/2)/2.
\end{align*}
Let $\mu$ be the measure $\frac{2(1-r^2)^{(k-2)/2}dr}{\vol(D)d\beta(d/2,k/2)}$.  Note that $\mu$ is the projected measure from the $d+k-1$-sphere onto the $d$-disk.  We have that $\int_D f^2(r)d\mu = \frac{2}{d\beta(d/2,k/2)}$.  Rescaling $f$ so that $$\int_D f(r)^2 d\mu = 1$$ we need to show that for such $f$, $\sup_{D} f = O\left(\frac{n}{d+k-1}\right)^{(d+k-1)/2}.$

Pulling $f$ back onto the $(d+k-1)$-sphere, we get that $\int_{S^{d+k-1}} f^2(x)dx = 1$, where $dx$ is the normalized measure on $S^{d+k-1}$.  We need to show that for $x\in S^d$ that $f(x) = O\left(\frac{n}{d+k-1}\right)^{(d+k-1)/2}.$  Let $\phi_i$ $(1\leq i \leq M)$ be an orthonormal basis of the space of polynomials of degree at most $n$ on $S^{d+k-1}$.  We can write $f(y) = \sum_i a_i \phi_i(y)$.  It must be the case that $\sum_i a_i^2=1$ and $f(x) = \sum_i a_i \phi_i(x)$.  By Cauchy Schwartz this is at most $\sqrt{\sum_{i=1}^M \phi_i^2(x)}$.  Consider the polynomial $\sum_{i=1}^M \phi_i^2(y)$.  This is clearly invariant under $SO(d+k)$ (since it is independent of the choice of basis $\phi_i$).  Therefore this function is constant.  Furthermore its average value on $S^{d+k-1}$ is clearly $M$.  Therefore $f(x) \leq \sqrt{M}$.

On the other hand we have that
$$
M = \binom{n+d+k-1}{d+k-1} + \binom{n+d+k-2}{d+k-1} = O\left( \frac{n}{d+k-1}\right)^{d+k-1}.
$$
This completes our proof.
\end{proof}

\begin{lem}\label{PositiveFourierLem}
Let $f$ be a real-valued polynomial of degree at most $n$ on $S^1$.  Suppose that $f \geq 0$ on $S^1$.  We can write $f$ in terms of a Fourier Series as $$
f(\theta) = \sum_{k=-n}^n a_k e^{ik\theta}.
$$
Then $a_0$ is real and $a_0 \geq |a_k|$ for all $k$.
\end{lem}
\begin{proof}
The fact that $f$ can be written in such a way comes from noting that $e^{\pm ik\theta}$ are the spherical harmonics of degree $k$ on $S^1$.  Since $f$ is real valued it follows that $a_{-k} = \bar{a_k}$ for all $k$.  We have that
$$
a_0 = \frac{1}{2\pi} \int_0^{2\pi} f(\theta)d\theta = \frac{1}{2\pi} \int_0^{2\pi} |f(\theta)e^{-ik\theta}|d\theta \geq \left|\frac{1}{2\pi}\int_0^{2\pi} f(\theta)e^{-ik\theta}d\theta\right| = |a_k|.
$$
\end{proof}

\begin{lem}\label{VariationOnCircleLem}
If $f$ is a polynomial of degree at most $n$ on $S^1$, and if $f$ is non-negative on $S^1$, then $\var_{S^1}(f) = O(n)\int_{S^1} f.$
\end{lem}
\begin{proof}
Consider $f=f(\theta)$ as above.  For an angle $\phi$, let $g_\phi(\theta) = {f(\phi+\theta)+f(\phi-\theta)}$.  Clearly $g_\phi$ is non-negative, and $\int_{S^1} g_\phi = 2\int_{S^1} f$.  Furthermore, we have that
\begin{align*}
\int_0^{2\pi} \var_{S^1} (g_\phi) d\phi & = \int_0^{2\pi} \int_0^{2\pi} |f'(\phi+\theta) - f'(\phi-\theta)| d\theta d\phi\\
& = \int_0^{2\pi} \int_0^{2\pi} |f'(\vartheta) - f'(\rho)|d\rho d\vartheta \\
& \geq 2\pi \int_0^{2\pi} |f'(\vartheta)| d\vartheta\\
& = 2\pi \var_{S^1}(f).
\end{align*}
Where above we use the fact that $\int_0^{2\pi} f'(\rho)d\rho = 0$ and that the absolute value function is convex.  Hence for some $\phi$, $\var_{S^1}(g_\phi) \geq \var_{S^1}(f)$.  Therefore, we may consider $g_\phi$ instead of $f$.  Noting that $g_\phi(\theta) = g_\phi(-\theta)$, we find that $g_\phi$ can be written as $p(\cos\theta)$ for some polynomial $p$ of degree at most $n$.  Our result then follows from Lemma \ref{circleVariationLem}.
\end{proof}

\begin{lem}\label{HalfDesignLem}
Let $d\geq 0$ be an integer.  Consider the design problem given by $X=[0,1]$, $\mu = r^d dr/(d+1)$, and $W$ the set of polynomials of degree at most $n$ in $r^2$.  Then there exists a weighted design for this problem, $(w_i,r_i)$ where $w_i = \Omega_d(r_i^d \sqrt{1-r_i^2}n^{-1})$, $\min(r_i) = \Omega(n^{-1})$, and $\max(r_i) = 1-\Omega(n^{-2})$.
\end{lem}
\begin{proof}
For any such polynomial $p(r^2)$ we have that
$$
\int_0^1 p(r^2) r^d/(d+1) dr = \frac{1}{2(d+1)}\int_0^1 p(s) s^{(d-1)/2} ds.
$$
Therefore, if we have a weighted design $(w_i,s_i)$ for the design problem $([0,1],\frac{s^{(d-1)/2}ds}{2(d+1)},\mathcal{P}_n)$, then $(w_i,\sqrt{s_i})$ will be a weighted design for our original problem.  We use the design implied by Lemma \ref{IntervalWeightedDesignLem}.  The bound on the $w_i$ is implied by Equation \ref{RBoundEqn}.  The bounds on the endpoints are implied by our observation that there are no roots of $P_n^{((d-1)/2,0)}$ within $O_d(n^{-2})$ of either endpoint.
\end{proof}

We are now ready to prove Proposition \ref{SphereGraphProp}.  We prove by induction on $d\geq 1$ that for any $n$, there exists a graph $G^d_n$ on $S^d$ with $K_{G^d_n} = O_d(n^d \log(n)^{d-1})$ and so that the total length of the edges of $G^d_n$ is $n^{O_d(1)}$.  For $d=1$, we let $G^d_n = S^1$.  This suffices by Lemma \ref{VariationOnCircleLem}.  From this point on, all of our asymptotic notation will potentially depend on $d$.

In order to construct these graphs for larger $d$, we will want to pick a convenient parametrization of the $d$-sphere.  Consider $S^d\subset \R^{d+1}$ as $\{x: |x|=1\}$.  We let $r$ be the coordinate on the sphere $\sqrt{\sum_{i=1}^{d-1} x_i^2}$.  We let $u\in S^{d-2}$ be the coordinate so that $(x_1,x_2,\ldots,x_{d-1})=ru$.  We let $\theta$ be the coordinate so that $(x_d,x_{d+1}) = \sqrt{1-r^2}(\cos\theta,\sin\theta)$.  Note that $u$ is defined except where $r=0$ and $\theta$ is defined except where $r=1$.  Note that in these coordinates, the normalized measure on $S^d$ is given by $\frac{r^{d-2}drdud\theta}{2\pi (d-1)}$. We also note that if $\phi^m_i$ are an orthonormal basis for the degree $m$ spherical harmonics on $S^{d-2}$, that an orthonormal basis for the polynomials of degree at most $n$ on $S^d$ is given by
$$
(1-r^2)^{k/2}e^{ik\theta}r^m \phi^m_i(u) P^{k,m,d}_\ell(r^2)
$$
Where $k,m,\ell$ are integers with $m,\ell\geq 0$ and $|k|+m+2\ell \leq n$ and where the $P^{k,m,d}_\ell(r^2)$ are orthogonal polynomials for the measure $r^{m+d-2}(1-r^2)^{k/2} dr/(d-1)$ on $[0,1]$ and functions in $r^2$, or, equivalently, $P^{k,m,d}_\ell(s)$ are the orthogonal polynomials for the measure $s^{(m+d-3)/2}(1-s)^{k/2}ds/(2(d-1))$ on $[0,1]$.

We construct $G^d_n$ as follows.  Our construction will depend on the graph given by our inductive hypothesis for $d-2$.  Since our Theorem does not hold for $d=0$, this means that our construction will need to be slightly altered in the case $d=2$.  On the other hand, there is a \emph{disconnected } graph, $G$ on $S^0$ with $K_G=O(1)$ that has total length $n^{O(1)}$ and supports a design of size $2$ (this graph of course being the union of two loops, one at each point of $S^0$).  This will turn out to be a sufficient inductive hypothesis to prove our $d=2$ case with only minor modification.  We now proceed to explain the construction of $G^d_n$.

Let $(w_i,r_i)$ $(1\leq i \leq h)$ be the design for the measure $r^{d-2}dr/(d-1)$ on $[0,1]$ for polynomials of degree at most $2n$ in $r^2$ as described in Lemma \ref{HalfDesignLem}.

We first consider the construction for $d>2$.  Let $N = An^{d-2}(\log(n))^{d-2}$ for $A$ a sufficiently large constant.  For each $r_i$, let $N_i = [r_i^{d-2} N]$ and $k_i = \left[\frac{Br_i n \log(n)}{\log(nr\log(n))}\right]$, where $B$ is a constant chosen so that both $B$ and $A/B$ are sufficiently large.  We inductively construct $G_i=G^{d-2}_{k_i}$.  By the inductive hypothesis for the design problem $\sd{d-2}{k_i}$, $K_{G_i} < (N_i)$ if $A$ was sufficiently large compared to $B$.  Therefore, by Proposition \ref{KGraphProp} there is a design $u_{i,j}$, $1\leq j \leq N_i$ for the design problem $\sd{d-2}{k_i}$ so that each of the $u_{i,j}$ lies on $G_i$.  Let $r_1$ be the smallest of the $r_i$.  By rotating $G_i,u_{i,j}$ if necessary we can guarantee that $r_iu_{i,1} = (r_1,\sqrt{r_i^2-r_1^2},0,\ldots,0)$ for all $i$.

We now define our graph $G=G^d_n$ as follows in $(r,u,\theta)$ coordinates.  First we define $H$ to be the union of:
\begin{itemize}
\item The circles $(r_i,u_{i,j},\theta)$ for $\theta\in [0,2\pi]$ for $1\leq i \leq h$ and $1\leq j \leq N_i$
\item The graphs $(r_i,u,0)$ for $u\in G_i$ for $1\leq i \leq h$
\end{itemize}
We note that $H$ is not connected.  Its connected components correspond to the $r_i$, since each $G_i$ connects all of the circles at the corresponding $u_{i,j}$.  We let $G=H\cup H'$, where $H'$ is the image of $H$ under the reflection that swaps the coordinates $x_2$ and $x_d$.  We note that $H$ union the circle in $H'$ corresponding to $u_{1,1}$ is connected.  Since this circle is parameterized as $(r_1,\sqrt{1-r_1^2}\sin\theta,0,0,\ldots,0,\sqrt{1-r_1^2}\cos\theta)$ intersects each of the $u_{i,1}$ in $H$.  Similarly $H'$ union the circle over $u_{1,1}$ in $H$ is connected.  Hence $G$ is connected.  It is also clear that the total length of all the edges of $G$ is $n^{O(1)}$.  We now only need to prove that $K_G = O(n^d\log(n)^{d-1})$.  We note that it suffices to prove that $K_H=O(n^d\log(n)^{d-1})$ since $K_G \leq K_H + K_{H'} = 2 K_H$.

For $d=2$, we need to make a couple of small modifications to the above construction.  The graphs $G^0_n$ are of course trivial.  In this case, it will be sufficient to let $N=N_i=2$ and $k_i=\left[\frac{Br_i n \log(n)}{\log(nr\log(n))}\right]$ for $B$ a sufficiently large constant.  We still have a design of size $N_i$ on $S^0$ (of unlimited strength) given by $\{-1,1\}$.  The graph $H$ is now given by a union of latitude lines of our sphere supported on the latitudes $\pm r_i$.  $H$ now has two connected components for each $r_i$ (instead of the one we see in other cases).  On the other hand, it is still the case that if $H'$ is the rotation of $H$ by $90\deg$, then the most central of the circles in $H'$ meets each connected component of $H$ (and visa versa), and hence $G=H\cup H'$ is connected.  The remainder of our argument will hold identically for the $d=2$ and $d>2$ cases.

Let $v_i = \frac{w_i}{N_i}$.  We note that $v_i = \Omega(n^{-1}N^{-1}\sqrt{1-r_i^2})$.  We claim that the circles in $H$ with weights given by $v_i$ form an approximate design in the following sense.

\begin{lem}\label{approximateDesignLem}
Let $C$ be any real number.  Then if $B/C$ is sufficiently large, and $f\in\mathcal{P}_{4n}$ we have that
\begin{equation}\label{approximateDesignEqn}
\left|\int_{S^d} f - \sum_{i,j} v_i \frac{1}{2\pi} \int_0^{2\pi} f(r_i,u_{i,j},\theta)d\theta \right| = O(n^{-C})|f|_2.
\end{equation}
\end{lem}
\begin{proof}
We note that after increasing $C$ by a constant, it suffices to check our Lemma for $f$ in an orthonormal basis of $\mathcal{P}_{2n}$.  Hence we consider
$$
f(r,u,\theta)=(1-r^2)^{k/2}e^{ik\theta}r^m \phi^m(u) P^{k,m,d}_\ell(r^2)
$$
for $\phi^m$ some degree-$m$ spherical harmonic. Note that unless $k=0$, both of the terms on the left hand side of Equation \ref{approximateDesignEqn} are 0.  Hence we can assume that $k=0$ and
$$
f(r,u,\theta)=f(r,u) = r^m \phi^m(u) P^{m,d}_\ell(r^2).
$$
We need to show that
$$
\left|\int r^{m+d-2}\phi^m(u)P^{m,d}_\ell(r^2) \frac{drdu}{d-1} - \sum_{i,j} v_i r_i^m P^{m,d}_\ell(r_i^2)\phi^m(u_{i,j}) \right| = O(n^{-C}).
$$
First we note that if $m=0$, $\phi^m(u)=1$.  In this case
\begin{align*}
\sum_{i,j} v_i r_i^m P^{m,d}_\ell(r_i^2)\phi^m(u_{i,j}) & = \sum_i N_i v_i P^{m,d}_\ell(r_i^2)\\
& = \sum_i w_i P^{m,d}_\ell(r_i^2)\\
& = \int_0^1 r^{d-2} P^{m,d}_\ell(r^2)dr/(d-1)\\
& = \int_{S^d} f.
\end{align*}
Where we use above the fact that $w_i,r_i$ is a weighted design.  Hence we are done for the case $m=0$.

For $m>0$, the integral of $f$ over $S^d$ is 0.  Furthermore for $k_i\geq m$, $\sum_j \phi^m(u_{i,j}) = 0$ (since the $u_{i,j}$ are a design).  Therefore in this case, the left hand side of Equation \ref{approximateDesignEqn} is
$$
\left|\sum_{k_i < m}v_i r_i^m P^{m,d}_\ell(r_i^2)\sum_j \phi^m(u_{i,j}) \right|.
$$
By results in the proof of Lemma \ref{L2toMaxLem}, we have that $\phi^m(u_{i,j}) = n^{O_d(1)}$.  Furthermore $v_i = O(1)$ and there are $n^{O_d(1)}$ many pairs of $i,j$ in the sum.  Therefore, this is at most
$$
n^{O(1)}\max_{i:k_i < m} |r_i^m P^{m,d}_\ell(r_i^2)|.
$$
The fact that $|f|_2 =1$ implies that
\begin{align*}
1&  = \int_0^1 r^{2m+d-2}(P^{m,d}_\ell(r^2))^2 dr/(d-1)\\
&= \int_0^1 s^{m+(d-3)/2}(P^{m,d}_\ell(s))^2 ds/(2(d-1))\\
& \geq \frac{1}{2^{m+(d+1)/2}(d-1)}\int_{-1}^1 (1-x^2)^{m+(d-3)/2}(P^{m,d}_\ell(2x-1))^2dx.
\end{align*}
Therefore, since the degree of $P^{m,d}_\ell$ is at most $n$, by Lemma \ref{L2toMaxLem} on the 1-disc we have that
$$
\max_{[0,1]}P^{m,d}_\ell= O\left(\frac{n}{m} \right)^{m+(d-1)/2} = n^{O(1)}O\left(\frac{n}{m}\right)^{m}
$$
This means that if $m>k_i$ that $r_i^m P^{m,d}_\ell(r_i^2)$ is at most
$$
n^{O_d(1)}O\left(\frac{nr_i}{m}\right)^{m}.
$$
Since for $B$ sufficiently large, $O\left(\frac{nr_i}{k_i}\right)$ would be less than $\frac{1}{2}$, this is at most
$$
n^{O(1)}O\left(\frac{nr_i}{k_i}\right)^{k_i}.
$$

Hence we need to know that,
\begin{equation}\label{kLargeEnoughEquation}
n^{O(1)}O\left(\frac{nr_i}{k_i}\right)^{k_i} = O(n^{-C}).
\end{equation}
This holds because if $nr_i < \log(n)$ the left hand side of Equation \ref{kLargeEnoughEquation} is at most
$$
n^{O(1)}O(\log(n)^{-1/2})^{\Omega(B \log(n)/\log\log(n))} = n^{O(1)-\Omega(B)}.
$$
Where we use the fact that $nr_i = \Omega(1)$.  If on the other hand $nr_i \geq \log(n)$, then $k_i = \Omega(B\log (n))$ and the left hand side of Equation \ref{kLargeEnoughEquation} is
$$
n^{O(1)}O(B^{-1})^{\Omega(B\log(n))} = n^{O(1)-\Omega(B)}.
$$
This completes our proof.
\end{proof}

For $f$ a polynomial on $S^d$ let
$$
A(f) := \sum_{i,j} v_i \frac{1}{2\pi} \int_0^{2\pi} f(r_i,u_{i,j},\theta)d\theta.
$$
Let
$$
A_i(f) := \sum_j v_i \frac{1}{2\pi} \int_0^{2\pi} f(r_i,u_{i,j},\theta)d\theta.
$$
$$
A_{i,j}(f) := v_i \frac{1}{2\pi} \int_0^{2\pi} f(r_i,u_{i,j},\theta)d\theta.
$$

\begin{lem}\label{AL2Lem}
For $f\in \mathcal{P}_{2n}$, $f\geq 0$ on $H$,
$$
A(f^2) = n^{O(1)}A(f)^2
$$
\end{lem}
\begin{proof}
Since $f(r_i,u_{i,j},\theta)$ is a non-negative polynomial of degree at most $2n$ on the circle,
$$
\frac{1}{2\pi}\int f(r_i,u_{i,j},\theta)^2 d\theta = O(n)\left(\frac{1}{2\pi} \int f(r_i,u_{i,j},\theta)d\theta\right)^2.
$$
So $A_{i,j}(f^2) = O(n)A_{i,j}(f)^2$.
$$
A(f) = \sum_{i,j} v_j A_{i,j}(f)
$$
$$
A(f^2) = O(n)\sum_{i,j} v_i A_{i,j}(f)^2 \leq O(n)\sum_{i,j} v_i (A(f)/v_i)^2 = n^{O(1)}A(f)^2.
$$
Where the last equality holds since $v_i = \Omega(n^{-1}N^{-1}\sqrt{1-r_i^2})$, and $1-r_i^2 = \Omega(n^{-2})$ for all $i$.
\end{proof}

We now prove a more useful version of Lemma \ref{approximateDesignLem}.
\begin{lem}\label{ApproximationBoundLem}
If $B$ is sufficiently large, and if $f$ is a polynomial of degree at most $2n$ on $S^d$ that is non-negative on $H$ then
$$
|\int_{S^d} f - A(f)| \leq \frac{A(f)}{2}.
$$
\end{lem}
\begin{proof}
By Lemma \ref{approximateDesignLem} applied to $f^2$, we have that
$$
|f|_2^2 = n^{O(1)}A(f)^2 + O(n^{-C})|f^2|_2.
$$
On the other hand, we have that $\sup_{S^d}(|f|) = n^{O(1)}|f|_2$.  Therefore,
$$|f^2|_2^2 \leq |f|_2^2\sup_{S^d}(|f|)^2 \leq n^{O(1)} |f|_2^4.$$
Hence, we have that
$$
|f|_2^2 = n^{O(1)}A(f)^2 + n^{O(1)-C}|f|_2^2.
$$
If the above holds for sufficiently large $C$ (which by Lemma \ref{approximateDesignLem} happens if $B$ is sufficiently large), this implies that
$$
|f|_2^2 = n^{O(1)}A(f)^2,
$$
or that
$$
|f|_2 = n^{O(1)}A(f).
$$

Therefore, for $B$ sufficiently large, we have that
$$
|\int_{S^d} f - A(f)| \leq O(n^{-C})|f|_2 \leq n^{O(1)-C}A(f) \leq \frac{A(f)}{2}.
$$
\end{proof}

\begin{cor}\label{AUpperBoundCor}
Assuming $B$ is sufficiently large, if $f$ is a polynomial of degree at most $2n$ on $S^d$ and $f$ is non-negative on $H$ then
$$
A(f) \leq 2\int_{S^d} f.
$$
\end{cor}

We will now try to bound $K_H$ based on a variant of one of our existing criteria.  In particular, we would like to show that if $f$ is a degree $n$ polynomial with $\int f = 1$ and $f\geq 0$ on $H$ that $\var_G(f) = O(n^d\log(n)^{d-1})$.  Replacing $f$ by $\frac{f+1}{A(f+1)}$ and noting by Corollary \ref{AUpperBoundCor} that $A(f+1)\leq 4$, we can assume instead that $f\geq \frac{1}{4}$ on $H$ and that $A(f)=1$.

We first bound the variation of $f$ on the circles over $u_{i,j}$.  Define
$$
f_{i,j}(\theta) := f(r_i,u_{i,j},\theta).
$$

We will prove the following:

\begin{prop}\label{uCircleVariationBoundProp}
Let $B$ be sufficiently large.  Let $f$ be a degree $n$ polynomials with $f\geq 1/4$ on $H$ and $A(f)=1$.  Then
$$
\var_{S^1} f_{i,j} = O(n^d\log(n)^{d-1})A_{i,j}(f).
$$
\end{prop}

This would follow immediately if $f_{i,j}$ was degree at most $n\log(n) \sqrt{1-r^2}$.  We will show that the contribution from higher degree harmonics is negligible.

We define for integers $k$, $a_k(r,u)$ to be the $e^{ik\theta}$ component of $f$ at $(r,u,\theta)$.  We note that $a_k(r,u) = (1-r^2)^{|k|/2}P_k(\vec{r})$, where $\vec{r} = ru$ is a coordinate on the $(d-1)$-disc and $P_k(\vec{r})$ some polynomial.

We first show that $|a_k(r,u)|$ is small for $k > n\log(n) \sqrt{1-r^2}$.
\begin{lem}\label{FourierCoefBoundLem}
Let $C$ be a real number so that $B/C$ is sufficiently large.  Let $f$ be a degree $n$ polynomial with $f\geq 0$ on $H$ and $A(f)=1$.  Then for $|k|> n\log(n)\sqrt{1-r_i^2}$, $|a_k(r_i,u)| = O(n^{-C})$.
\end{lem}
\begin{proof}
We have that $|a_k|_2 \leq |f|_2 = n^{O(1)}$ by Lemma \ref{AL2Lem}.  Therefore,
$$
\int_{D^{d-1}} (1-r^2)^{|k|}P_k^2(\vec{r})dr = n^{O(1)}.
$$
Applying Lemma \ref{L2toMaxLem}, we find that
$$
|P_k(\vec{r})| \leq n^{O(1)}O\left(\frac{n}{|k|} \right)^{|k|/2}.
$$
Therefore,
$$
|a_k(r_i,u)| \leq n^{O(1)}O\left(\frac{n\sqrt{1-r_i^2}}{|k|} \right)^{|k|/2} \leq n^{O(1)} O\left(\frac{1}{\log(n)}\right)^{|k|/2}.
$$
Since $|k| = \Omega(\log(n))$ (because $\sqrt{1-r_i^2} = \Omega(n^{-1})$), this is $O(n^{-C})$.
\end{proof}

\begin{proof}[Proof of Proposition \ref{uCircleVariationBoundProp}]
Let $f_{i,j}^l$ be the component of $f_{i,j}$ coming from Fourier coefficients of absolute value at most $n\log(n)\sqrt{1-r_i^2}$.  By Lemmas \ref{MaxAndDerivativeBounds} and \ref{FourierCoefBoundLem}, we have that for $B$ sufficiently large, $f_{i,j}-f_{i,j}^l$ is less than $1/8$ everywhere and has Variation $O(1)$.  But since $f_{i,j}^l$ is non-negative and has bounded Fourier coefficients, we have by Lemma \ref{VariationOnCircleLem} that
$$
\var_{S^1}f_{i,j}^l = O\left(n\log(n)\sqrt{1-r_i^2}\right)\int_{S^1} f_{i,j}^l = O\left(n\log(n)\sqrt{1-r_i^2}\right)\int_{S^1} f_{i,j}.
$$
This means that
\begin{align*}
\var_{S^1}(f_{i,j}) & = O\left(\frac{n\log(n)\sqrt{1-r_i^2}}{v_i}\right)A_{i,j}(f)\\
& =  O\left(\frac{N_i n\log(n)\sqrt{1-r_i^2}}{w_i}\right)A_{i,j}(f)\\
& =  O\left(\frac{(r_in\log(n))^{d-2} n\log(n)\sqrt{1-r_i^2}n}{\sqrt{1-r_i^2}}\right)A_{i,j}(f)\\
& =  O\left(n^d\log(n)^{d-1}\right)A_{i,j}(f).
\end{align*}
\end{proof}

We now bound the variation of $f$ on the $G_i$ in $H$.

\begin{prop}\label{GiVariationBoundProp}
Suppose that $B$ is sufficiently large.  For $f\in \mathcal{P}_n$, $f\geq \frac{1}{4}$ on $H$, $A(f)=1$, $\var_{G_i}(f) \leq A_i(f)O(n^d\log(n)^{d-1}).$
\end{prop}

Again this would be easy if we knew that the restriction of $f$ to the appropriate sphere was low degree.  Our proof will show that the contribution from higher degree harmonics is small.

Let $f_i(u)=f(r_i,u,0)$ be $f$ restricted to the $(d-2)$-sphere on which $G_i$ lies.  We claim that the contribution to $f$ from harmonics of degree more than $k_i$ is small.  In particular we show that:

\begin{lem}\label{LargeHarmonicsBoundLem}
Let $C$ be a real number.  Suppose that $B/C$ is sufficiently large.  Let $f\in \mathcal{P}_n$, $f\geq 0$ on $H$, $A(f)=1$.  Let $f_i^h(u)$ be the component of $f_i$ coming from spherical harmonics of degree more than $k_i$.  Then $|f_i^h|_2 = O(n^{-C})$.
\end{lem}
\begin{proof}
Perhaps increasing $C$ by a constant, it suffices to show that for $\phi$ a spherical harmonic of degree $m>k_i$ that the component of $\phi$ in $f_i$ is $O(n^{-C})$.  We will want to use slightly different coordinates on $S^d$ than usual here.  Let $s=(x_d,x_{d+1})$ be a coordinate with values lying in the 2-disc.  The component of $f$ corresponding to the harmonic $\phi(u)$ is given by
$$
\phi(u) (1-s^2)^{m/2} Q(s)
$$
for $Q$ some polynomial of degree at most $n$.  Considering the $L^2$ norm of $f$, we find that
$$
\int_{D^2} (1-s^2)^m Q^2(s) ds \leq \pi |f|_2^2 \leq n^{O(1)}A(f)^2 = n^{O(1)}.
$$
Applying Lemma \ref{L2toMaxLem} to $Q(s)$, we find that $|Q(s)| = n^{O(1)}O\left(\frac{n}{m} \right)^{m/2}$.  Hence the component of $\phi$ at $r_i$ is $r_i^{m/2}Q(r_i,0)$, which is at most
$$
n^{O(1)}O\left( \frac{nr_i}{m} \right)^{m/2}.
$$
Since $m>k_i$, $\frac{nr_i}{m} < \frac{1}{2}$, the above is at most
$$
n^{O(1)}O\left( \frac{nr_i}{k_i} \right)^{k_i/2} = O(n^{-C})
$$
by Equation \ref{kLargeEnoughEquation}.
\end{proof}

We can now prove Proposition \ref{GiVariationBoundProp}.
\begin{proof}[Proof of Proposition \ref{GiVariationBoundProp}]
Let $f_i^l(u)$ be the component of $f_i$ coming from spherical harmonics of degree at most $k_i$.  By Lemmas \ref{LargeHarmonicsBoundLem} and \ref{MaxAndDerivativeBounds}, we have that for $B$ sufficiently large, $f_i^l \geq 0$ on $G_i$ and that $|\var_{G_i}(f)-\var_{G_i}(f_i^l)| \leq v_i/4 \leq A_i(f)$.  Hence it suffices to prove that $\var_{G_i}(f_i^l) = A_i(f)O(n^d\log(n)^{d-1})$.  Since for polynomials of degree at most $k_i$ on $S^{d-2}$, $K_{G_i} = O(k_i^{d-2} \log(k_i)^{d-2})$, we have that $\var_{G_i}(f_i^l) = O(k_i^{d-2} \log(k_i)^{d-2})\int_{S^{d-2}} f_i^l$.  Since the $u_{i,j}$ form a spherical design this is $$O(k_i^{d-2}\log(k_i)^{d-2}) \frac{1}{N_i} \sum_j f_i^l(u_{i,j}).$$  Again, for $B$ sufficiently large, this is $$O(k_i^{d-2}\log(k_i)^{d-2}) \frac{1}{N_i} \sum_j f(r_i,u_{i,j},0).$$
Now consider $F(\theta) = \frac{1}{N_i}\sum_j f(r_i,u_{i,j},\theta)$.  We have that $F$ is a polynomial of degree at most $n$ and that $F\geq 1/4$.  Let $F^l$ be the component of $F$ consisting of Fourier coefficients with $|k|\leq n\log(n)\sqrt{1-r_i^2}$.  By Lemmas \ref{MaxAndDerivativeBounds} and \ref{FourierCoefBoundLem}, if $B$ is sufficiently large, $|F-F^l|<1/8$.
It is clear that
$$A_i(f) = w_i \frac{1}{2\pi} \int_0^{2\pi} F(\theta)d\theta = \Theta(w_i)\frac{1}{2\pi} \int_0^{2\pi} F^l(\theta)d\theta.$$
Note that by Lemma \ref{VariationOnCircleLem}
\begin{align*}
F(0) & = O(1)+F^l(0)\\
& \leq \inf_{S^1}(F^l) + \var_{S^1}(F^l-\inf_{S^1}(F^l))\\
& = O(n\log(n)\sqrt{1-r_i^2}) \int_{S^1} F^l\\
& = O(n\log(n)\sqrt{1-r_i^2}) \int_{S^1} F.
\end{align*}
Therefore, we have that
\begin{align*}
\var_{G_i}(f) & = O(k_i^{d-2}\log(k_i)^{d-2}) F(0)\\
& = O(n\log(n)\sqrt{1-r_i^2} k_i^{d-2}\log(k_i)^{d-2}) A(f)/w_i\\
& = A(f) O\left(\frac{n \log(n)\sqrt{1-r_i^2}k_i^{d-2}\log(k_i)^{d-2}}{w_i}\right)\\
& = A(f) O\left(\frac{n\log(n)\sqrt{1-r_i^2}(r_in\log(n))^{d-2}}{r_i^{d-2}n^{-1}\sqrt{1-r_i^2}}\right)\\
& = A(f) O(n^d\log(n)^{d-1}).
\end{align*}
\end{proof}

We can finally prove Proposition \ref{SphereGraphProp}.
\begin{proof}
We proceed by induction on $d$.  For $d=1$ the $S^1$ suffices as discussed.  Assuming that we have the graph for $d-2$ we construct $G$ as described above.  Clearly $G$ is connected and has total length $n^{O(1)}$.  We need to show that $K_H = O(n^d\log(n)^{d-1}).$  To do so it suffices to show that for any $f\in \mathcal{P}_n$ with $f\geq 1/4$ on $H$ and $A(f)=1$ that $\var_H(f) = O(n^d\log(n)^{d-1})$.  We have that
\begin{align*}
\var_H(f) & = \sum_{i,j} \var_{S^1}(f_{i,j}) + \sum_i \var_{G_i} f_i\\
& = O(n^d\log(n)^{d-1})\left(\sum_{i,j} A_{i,j}(f) + \sum_i A_i(f) \right)\\
& = O(n^d\log(n)^{d-1})(A(f)+A(f))\\
& = O(n^d\log(n)^{d-1}).
\end{align*}
This completes the proof.
\end{proof}

Theorem \ref{SphereicalDesignThm} now follows from Proposition \ref{SphereGraphProp} and Proposition \ref{KGraphProp}.

\section{Acknowledgements}

This work was done while the author was an intern at Microsoft Research.

\end{document}